\documentclass{article}
\usepackage{amssymb}
\usepackage{color}
\usepackage{amsthm}
\usepackage{mhsetup}
\usepackage{mathtools}
\mathtoolsset{showonlyrefs} 
\numberwithin{equation}{section}
\newcommand{\RR}{\mathbb{R}}
\newcommand{\NN}{\mathbb{N}}
\newcommand{\ZZ}{\mathbb{Z}}
\newcommand{\dist}{\operatorname{dist}}
\newtheorem{Theorem}{Theorem}[section]

\newtheorem{Proposition}[Theorem]{Proposition}
\newtheorem{Lemma}[Theorem]{Lemma}
\theoremstyle{definition}
\newtheorem{Definition}[Theorem]{Definition}
\newtheorem{Assumption}[Theorem]{Assumption}
\theoremstyle{remark}
\newtheorem{Remark}[Theorem]{Remark}

\begin{document}
\title{Periodic solutions to relativistic Kepler problems:\\
a variational approach\thanks{Under the auspices of INdAM-GNAMPA, Italy.}
}
\author{A. Boscaggin, W. Dambrosio and D. Papini}
\date{}
\maketitle
\begin{abstract}
We study relativistic Kepler problems in the plane. At first, using non-smooth critical point theory, we show that under a general time-periodic external force of gradient type there are two infinite families of $T$-periodic solutions, parameterized by their winding number around the singularity: the first family is a sequence of local minima, while the second one comes from a mountain pass-type geometry of the action functional. Secondly, we investigate the minimality of the circular and non-circular periodic solutions of the unforced problem, via Morse index theory and level estimates of the action functional.
\end{abstract}
\medskip



\section{Introduction}

The motion of a relativistic particle in a Kepler potential can be described by the equation
\begin{equation}\label{eq:kep}
\frac{d}{dt}\left(\frac{m\dot{x}}{\sqrt{1-|\dot{x}|^{2}/c^{2}}}\right) =
-\alpha\frac{x}{|x|^{3}}, \qquad x \in \RR^{2}\setminus\{0\},
\end{equation}
where $\vert \cdot \vert$ stands for the Euclidean norm of a two-dimensional vector, $m$ is the mass of the particle, $c$ is the speed of light and $\alpha >0$ is a constant. Indeed, equation \eqref{eq:kep} is the Euler-Lagrange equation of the Lagrangian 
$$
\mathcal{L}(x,\dot x) = K(\dot x) + V(x),
$$
where 
$$
K(\dot x) = mc^2 \left( 1 - \sqrt{1-\frac{\vert \dot x \vert^2}{c^2}}\right)
$$
is the relativistic kinetic energy and 
$$
V(x) = \frac{\alpha}{\vert x \vert}
$$ 
is the usual Newtonian gravitational potential. Such a model, providing the simplest relativistic correction for the classical Kepler problem and therefore usually called relativistic Kepler problem, has been proposed and analyzed in several papers and books
of mathematical physics (see, among others, \cite{AnBa71,Bo04,DeEr85,GoPoSa02,Ji13,LeMo-pp,MuPa06} and the references therein).

The interest for problem \eqref{eq:kep} from a mathematical analysis perspective seems to be much more recent and, to the best of our knowledge, only very few references can be quoted, all of them dealing with the existence of periodic solutions for different time-dependent perturbations of \eqref{eq:kep}. More precisely, in the papers \cite{ToUrZa13, Za13} the equation 
$$
\frac{d}{dt}\left(\frac{m\dot{x}}{\sqrt{1-|\dot{x}|^{2}/c^{2}}}\right) =
-M(t)\frac{x}{|x|^{3}},
$$
where $M$ is a positive and periodic function, is considered and the existence of periodic (and quasi-periodic) solutions is proved via topological degree theory. On the other hand, in \cite{Ga19} an abstract perturbation theorem of variational nature is used to establish the existence of nearly-circular periodic solutions for the perturbed problem
\begin{equation}\label{eq:keppert}
\frac{d}{dt}\left(\frac{m\dot{x}}{\sqrt{1-|\dot{x}|^{2}/c^{2}}}\right) =
-\alpha\frac{x}{|x|^{3}} + \varepsilon \, \nabla_x U(t,x),
\end{equation}
with $U$ a periodic function in its first variable, when the parameter $\varepsilon$ is small enough. Finally, in the recent paper \cite{BoDaFe22} a higher-dimensional version of the Poincar\'e-Birkhoff fixed point theorem is applied to \eqref{eq:keppert} to ensure the existence, again for $\varepsilon$ sufficiently small, of periodic solutions bifurcating from non-circular solutions of the unperturbed problem. More precisely, according to \cite[Theorem 1.1]{BoDaFe22}, a threshold $T^* > 0$ can be given so that, whenever the period $T$ of the external potential $U$ is larger than $T^*$, equation \eqref{eq:keppert} has $T$-periodic solutions with any winding number $k \geq 2$, provided that $\vert \varepsilon \vert < \varepsilon^*(k)$. 

Especially motivated by the results in \cite{BoDaFe22}, in this paper we continue the investigation about periodic solutions of relativistic Kepler problems, in two different, yet related, directions. 
\medbreak
In the first part of the paper, we investigate the existence of $T$-periodic solutions for the forced problem
\begin{equation}\label{eq:kepleroforzato}
\frac{d}{dt}\left(\frac{m\dot{x}}{\sqrt{1-|\dot{x}|^{2}/c^{2}}}\right) =
-\alpha\frac{x}{|x|^{3}} + \nabla_{x}U(t,x), 
\end{equation}
again with $U$ a potential which is $T$-periodic in the time variable. Our result, which can be seen as a sort of non-perturbative generalization of the ones in \cite{BoDaFe22,Ga19}, reads as follows.
\begin{Theorem}\label{thmain}
Let $ U:\RR\times\RR^{2}\to\RR$ be a potential which is smooth enough and $T$-periodic with respect to the first variable.
Then, for each $k\in\ZZ$ with $k\ne 0$, there exist at least two $T$-periodic solutions of \eqref{eq:kepleroforzato} that have winding number $k$ around the origin.
\end{Theorem}
In particular, problem \eqref{eq:kepleroforzato} has infinitely many $T$-periodic solutions: such a conclusion is new also in a perturbative setting, since it does not follow from the results in \cite{BoDaFe22,Ga19}. The precise smoothness assumption on $U$ will be given later, see Assumption \ref{ass:U}; notice however that no other hypoheses on the external potential $U$ (and on its period $T$) are needed. We also mention that more general equations, like for instance a forced relativistic $N$-center problem, could be considered, see Remark \ref{Ncentri}.

The proof of Theorem \ref{thmain} is of variational nature, relying on the fact that problem \eqref{eq:kepleroforzato} is, formally, the Euler-Lagrange equation of the action functional
\[
I(x) = \int_0^T \mathcal{L}(x,\dot x)\,dt + \int_0^T U(t,x) \,dt,
\]
where $\mathcal{L}(x,\dot x) = K(\dot x) + V(x)$ is, as already observed, the Lagrangian of the (unforced) relativistic Kepler problem.
More precisely, by following a well-established point of view since the pioneering paper by Gordon \cite{Go77}, one of the $T$-periodic solutions of Theorem \ref{thmain} is found by minimizing the functional $I$ on the class of $T$-periodic loops with winding number equal to $k$; the second solution, on the other hand, is provided by a min-max argument, relying again on the topology of the punctured plane.

To rigorously implement this strategy, two main difficulties have to be faced. 
The first one, of course, is due to the singular nature of the Kepler potential $V(x) = \alpha/\vert x \vert$, which in principle can lead, as it is typical in problems of celestial mechanics, to the occurrence of collisions for a minimum point of the functional $I$.
The second difficulty comes from the non-differentiability of the kinetic part of the relativistic Keplerian lagrangian, $K(\dot x) = -mc^2 (1-\vert \dot x \vert^2/c^2)^{1/2}$, making the functional $I$ not smooth on any natural space of functions and the usual critical point theory not directly applicable. To overcome these difficulties, we borrow some results from the recent paper \cite{ArBeTo20}, where 
a variational formulation is provided for the Lorentz force equation
\begin{equation}\label{lorentzintro}
\frac{d}{dt}\left(\frac{\dot{x}}{\sqrt{1-|\dot{x}|^{2}}}\right) =
E(t,x) + \dot x \wedge B(t,x), \qquad x\in \mathbb{R}^3,
\end{equation}
with $E$ and $B$ smooth functions having the role of electric and magnetic fields, respectively. The key point for this variational formulation is the choice of the Sobolev space $W^{1,\infty}_{T}$ of  Lipschitz continuous and $T$-periodic functions as the domain for the associated action functional, allowing for the use of Skzulin's version \cite{Sz86} of non-smooth critical point theory. As carefully explained in the introduction of \cite{ArBeTo20}, the choice of this functional space is forced by the presence of the term $\dot x \wedge B$ in equation \eqref{lorentzintro}.
However, quite surprisingly, it turns out to be very convenient also when applied to equation \eqref{eq:kepleroforzato}, which corresponds (after normalization of constants) to \eqref{lorentzintro} for $B \equiv 0$ but, on the other hand, presents a singularity in the term $E(t,x) = - \alpha x / \vert x \vert^3 + \nabla_x U(t,x)$ (the fact that $x$ is two or three-dimensional does not play a role).   
Indeed, if $x \in W^{1,\infty}_T$, the Keplerian term
$$
\int_0^T V(x) \,dt = \int_0^T \frac{\alpha}{\vert x \vert} \,dt
$$ 
of the relativistic action functional is finite if and only if $x$ never vanishes; moreover, with some more care, the behavior of the functional $I$ along sequences $\{x_n\} \subset W^{1,\infty}_T$ approaching the singularity can be successfully controlled: exclusion of collisions thus turns out to be much simpler than in the non-relativistic regime, cf. \cite{BoDaPa20}. 

Using these ideas, the existence of the minimal solution is easily given. Obtaining the second solution requires a considerable amount of extra-work. Indeed, the natural min-max principle to be used is a non-smooth version of the celebrated Ghossoub's min-max principle \cite{Gh93,GhPr89}, which however, to the best of our knowledge, has been established only when the smooth part of the functional is globally defined (see \cite{LiMa04}). We thus provide an extension of this result which is suited for our more general setting, see Theorem \ref{thm:minmax}; we also refer to the discussion after the statement for more comments.
A second difficulty comes from the fact that, as already observed in \cite{ArBeTo20}, the usual Palais-Smale condition is not available for the action functional in the space $W^{1,\infty}_{T}$. Indeed, our Palais-Smale sequences are precompact only with respect to the (weaker) $L^\infty$ topology and, as a consequence, a careful argument has to be used when the min-max level coincides with the minimal level.

\medbreak
The second part of the paper is again of variational nature, but it is concerned with the unforced problem \eqref{eq:kep}.
Of course, Theorem \ref{thmain} can be applied also in this situation (corresponding to $U \equiv 0$ of problem \eqref{eq:kepleroforzato}), thus providing, for any integer $k \neq 0$, a $T$-periodic solution to  \eqref{eq:kep}, obtained as a minimum point for the action functional $I$ among the $T$-periodic loops with winding number $k$ around the origin (incidentally, notice that here, due to the autonomous nature of the equation, the second solution is nothing but a time-translation of this minimal solution, see Remark \ref{rem:continuo}). On the other hand, as shown in \cite{BoDaFe22}, problem \eqref{eq:kep} has both circular and non-circular periodic solutions, and a complete picture for their existence and classification (depending on $T$ and $k$) can be provided: thus, it is natural to wonder which one of these solutions turns out to be the minimal one provided by Theorem \ref{thmain}.  

The complete answer to this question is provided by Theorem \ref{teo:minimalita}; here, we limit ourserlves to summarize the main conclusion in the following simplified form (we assume $k > 0$, since negative values of $k$ give rise to the same orbits, just run in the clockwise sense).

\begin{Theorem}\label{thmain2}
Let $k \geq 1$ be an integer. Then, the $T$-periodic solution of problem \eqref{eq:kep} minimizing the action functional 
(in the class of $T$-periodic loops with winding number $k$ around the origin) is the (unique, up to time-translations) circular periodic solution of minimal period $T/k$ whenever no other $T$-periodic solutions with winding number $k$ exist; otherwise, it is a non-circular solution and, precisely, is the (unique, up to time-translations and space-rotations) non-circular periodic solution with winding number $k$ and such that 
$\vert x \vert$ has minimal period equal to $T$.
\end{Theorem}

To better understand this statement, it can be useful to say (see Proposition~\ref{pro:circolari} and Proposition~\ref{pro:rosette}) that problem \eqref{eq:kep} possesses circular periodic solutions of any minimal period (and, thus, circular $T$-periodic solutions with winding number $k$ for any $T> 0$ and $k \geq 1$) and non-circular periodic solutions with winding number $k \geq 2$ if and only if $T > T^*$, with $T^* > 0$ the value found in \cite{BoDaFe22} already mentioned at the beginning of this introduction. Thus, the minimal solution is circular if and only if $k = 1$ or $k \geq 2$ and $T \leq T^*$, and non-circular otherwise. Moreover, since non-circular $T$-periodic solutions can be further distinguished via the minimal period of their radial component $r = \vert x \vert$ (which in general can be a submultiple of $T$), Theorem \ref{thmain2} further precises that the minimal non-circular solution is the one for which the minimal period of $r$ is exactly $T$.

The proof of Theorem \ref{thmain2} relies on the following strategy. At first (see Theorem \ref{teo:morse}) an explicit formula for the
Morse index of the circular solution is provided: such a formula is found by first computing the Conley-Zehnder index of the associated Hamiltonian system (as similarly done in \cite{KaOfPo21} for the classical Keplerian orbits) and then using the Morse index theorem, in the version for general Lagrangian systems given in \cite{Abb03}. As a consequence of this analysis, the Morse index of the circular solution is zero only when $k = 1$ or $k \geq 2$ and $T \leq T^*$. In these cases, since there are no non-circular solutions, the circular one is minimal.
On the other hand, for $k \geq 2$ and $T > T^*$ the minimal solution has to be non-circular. To establish which one of the non-circular solutions is the minimal one (recall that, in general, multiple non-circular solutions with winding number $k$ exist, being distinguished by the minimal period of their radial component $\vert x \vert$), a direct action level comparison is performed (see Proposition \ref{prop:azionecrescente}), implying that the solution with lowest action is the one such that $\vert x \vert$ has minimal period equal to $T$. It seems natural to conjecture that the Morse index of the non-minimal non-circular periodic solutions with winding number $k$ is related to the minimal period of their radial component, but a proof of this fact seems to be hard. 

The plan of the paper is the following. In Section \ref{sec2} the abstract variational setting is presented and the statement of the min-max principle is given: its proof, which is closely adapted from the one in \cite{LiMa04}, is developed in the Appendix. Then, in Section \ref{sec3} the proof of Theorem \ref{thmain} is provided. Finally, in Section \ref{sec4} we describe the periodic solutions of the unforced problem \eqref{eq:kep} and we give the proof of Theorem \ref{thmain2}.
\section{The abstract variational setting}\label{sec2}
Let $X$ be a Banach space. According to the seminal paper by Szulkin \cite{Sz86}, we are going to consider functionals $I: X \to (-\infty,+\infty]$ that are the sum of a convex functional $\psi$ and a smooth one $\Phi$; however, due to the presence of a singularity in equation \eqref{eq:kepleroforzato}, we will allow the functional $\Phi$ to be singular in a suitable sense. Even if assuming $\Phi$ of class $C^1$ would be enough for our application, for the sake of generality throughout this section we just suppose that $\Phi$ is locally Lipschitz continuous, as in the theory developed by Chang \cite{Chang81}. Summing up, we introduce the following basic set of assumptions.
\begin{Assumption}\label{ass:I}
$ I:X\to (-\infty,+\infty] $ is a functional which can be decomposed as
\[
I(x)=\psi(x)+\Phi(x), \quad \forall \ x\in X,
\]
where, denoting by $D_{\psi}=\{x\in X: \psi(x)<+\infty\}$ and $D_{\Phi}=\{x\in X: \Phi(x)<+\infty\}$,
\begin{enumerate}
\item
$D_{\Phi}$ is open in $X$ and $D_{I}=D_{\Phi}\cap D_{\psi}\ne\emptyset;$
\item
$\psi:X\to\RR\cup\{+\infty\} $ is convex and lower semi-continuous;
moreover, $\psi$ is continuous on any nonempty compact set $A\subset X$ such that $ \sup_{A} \psi $
is finite;
\item
$\Phi:X\to\RR\cup\{+\infty\} $ is locally Lipschitz continuous in $D_{\Phi}$, i.e. every $x\in D_{\Phi}$
has a neighbourhood $V_{x}\subset D_{\Phi}$ in which $\Phi$ is Lipschitz continuous;
\item for any sequence $\{x_n\}$ in $D_{I}=D_{\Phi} \cap D_{\psi}$ such that $\dist(x_n, \partial D_{\Phi}) \to 0$, it holds that
$I(x_n) \to +\infty$.
\end{enumerate}
The last condition, controling the behavior of the functional near the boundary of the set $D_{\Phi}$, is the crucial one to use tools from critical point theory also in the case $D_{\Phi} \neq X$. We also observe that, by the convexity of $\psi$, $ D_{\psi}$ is a convex set.
\end{Assumption}

\noindent
For functionals that satisfy a Lipschitz condition near a
point $ x\in D_{\Phi} $, like $ \Phi$ in Assumption~\ref{ass:I}, it is possible (see \cite[\S 2.1]{Clarke90}) to define a directional derivative $\Phi^{0}(x,u)$ at $x$
with respect to any direction $u\in X$ by setting
\[
\Phi^{0}(x;u) = \limsup_{ w\to x,t\to 0^{+}} \frac{\Phi(w+tu)-\Phi(w)}{t}.
\]
Of course, if $\Phi$ is of class $C^1$, then
\[
\Phi^{0}(x;u) = \Phi'(x)[u],\quad \forall \ x\in D_{\Phi},\ u\in X.
\]
We now recall some basic definitions from \cite[\S 3.2]{MoPa99}.
\begin{Definition}\label{def:puntocritico}
Let $ I : X\to (-\infty,+\infty] $ satisfy Assumption~\ref{ass:I}.
\begin{enumerate}
\item
A point $x \in D_{I}$ is a critical point of $I$ if
\[
\Phi^{0}(x;z-x) + \psi(z) - \psi(x) \ge 0, \quad \forall \ z \in X,
\]
\item
A \emph{Palais-Smale} (abbreviated \emph{PS-}) \emph{sequence for $I$ at level $c$}
is a sequence $\{x_{n}\}$ in $X$ such that $I(x_{n})\to c $ and
\[
\Phi^{0}(x_{n};z-x_{n})+\psi(z)-\psi(x_{n}) \ge-\epsilon_{n}\|z-x_{n}\|, \quad \forall \ n\in\NN
\text{ and } z\in X,
\]
for some sequence $\epsilon_{n}\to 0^{+}$.
\end{enumerate}
\end{Definition}
\begin{Proposition}\label{pro:minimoPScritico}
Let $ I : X\to (-\infty,+\infty] $ satisfy Assumption~\ref{ass:I}.
Then:
\begin{enumerate}
\item 
if $x\in D_{I}$ is a local minimum of $I$, then $x$ is a critical point of $I$;
\item
if $x\in D_{I}$ is a limit point of a PS-sequence for $I$ at level $c\in\RR$, then $x$ is a critical
point of $I$ and $I(x)=c$.
\end{enumerate}
\end{Proposition}
\begin{proof}
1. If $x\in D_{I}$ is a local minimum of $I$, the convexity of $\psi$ implies that
\begin{align*}
0 &\le I(x+t(z-x))-I(x) \\
& \le \Phi(x+t(z-x))-\Phi(x)+(1-t)\psi(x)+t\psi(z)-\psi(x) \\
& = \Phi(x+t(z-x))-\Phi(x)+t(\psi(z)-\psi(x))
\end{align*}
for all $z\in X$ and all positive $t$ which are small enough.
Hence the thesis follows by dividing by $t$ and recalling Definition~\ref{def:puntocritico}.

2. Without loss of generality, we may assume that $ \{x_{n}\} $ is a PS-sequence for $I$ at level $c$ such that
$ x_{n} \to x \in X $.
Assumption~\ref{ass:I}.4 and the fact that $ I(x_{n}) \to c $ grant that
$ x \not\in \partial D_{\Phi} $ and, thus, $ x \in D_{\Phi}$.
Since $ \Phi^{0} $ is upper semi-continuous (see \cite[Proposition~2.1.1]{Clarke90}) and $ -\psi $ is upper semi-continuous, we have
\[
\Phi^{0}(x;z-x) + \psi(z) - \psi(x) \ge \limsup_{n\to\infty} \left[\Phi^{0}(x_{n};z-x_{n}) + \psi(z) - \psi(x_{n}) \right]
\ge 0, \quad \forall \ z \in X.
\]
Choosing $ z=x $ in Definition~\ref{def:puntocritico}.2, we obtain
\[
\psi(x) \le \liminf_{n\to\infty} \psi(x_{n}) \le \limsup_{n\to\infty} \psi(x_{n})
\le \Phi^{0}(x_{n},x-x_{n})+\psi(x)+\epsilon_{n}\|x-x_{n}\| = \psi(x)
\]
and, hence, $I$ is continuous at $x$.
\end{proof}

\noindent
In our application of Section \ref{sec3}, we will find two families of critical points for $I$: the first one will be made by local minima.
In order to find the second one, we will use the following non-smooth version of the celebrated Ghossoub's min-max principle \cite{Gh93,GhPr89}.
\begin{Theorem}\label{thm:minmax}
Let $I=\psi+\Phi$ be a functional satisfying Assumption~\ref{ass:I}, let $B$ be a closed set in $X$ and $\mathcal{F}$ be a family of compact sets in $X$ such that:
\begin{enumerate}
\item 
$\mathcal{F}$ is \emph{homotopy stable with extended boundary} $B$, that is, for each
$ A \in \mathcal{F} $ and each continuous deformation $ \eta\in C^{0}([0,1]\times X, X) $ such that
$$ \eta(t,x)=x, \quad \forall \ (t,x)\in (\{0\}\times X)\cup ([0,1]\times B)  \qquad \text{ and } \qquad
\eta([0,1]\times A)\subset D_{\Phi},$$ 
one has that
$ \eta(\{1\}\times A)\in\mathcal{F} $;
\item
$ c \vcentcolon= \adjustlimits\inf_{A\in\mathcal{F}}\sup_{x\in A} I(x) < +\infty $;
\item
there exists a closed set $F$ in $X$ such that
\[
(A\cap F)\setminus B \ne \emptyset, \quad \forall \ A\in\mathcal{F} \qquad\text{and}\qquad
\sup_{B} I \le \inf_{F} I;
\]
\end{enumerate}
Then, for any sequence $\{A_{n}\}$ in $\mathcal{F}$ such that
$\adjustlimits \lim_{n\to\infty}\sup_{A_{n}} I = c$, there exists a PS-sequence $\{x_{n}\}\subset X$
at level $c$ such that $ \dist(x_{n},A_{n})\to 0 $.
If moreover $ \inf_F I = c $, then also $\dist(x_{n},F)\to 0$.
\end{Theorem}

\noindent
When $D_{\Phi} = X$, the result above has been established by Livrea and Marano in 
\cite[Theorem~3.1]{LiMa04}.
Here, we provide an extension to the case 
$D_{\Phi} \ne X$, having in mind the peculiar topology induced by the singularity
of the gravitational force in \eqref{eq:kepleroforzato}.
We show in Lemma~\ref{proplambdak} that our $D_{\Phi}$ is made up by countably many connected components $\Lambda_{k}$, $k\in\ZZ$, which are open in $X$ and, moreover,
$\Lambda_{k}\cap D_{\psi}$ is non-empty and bounded for each $k\ne 0$. 
For each $k\ne 0$, in Section~\ref{sec3.2} we find a minimum point $\bar{x}$ of $I$ in $\Lambda_{k}$: more precisely, $\bar{x}$ is a $T$ periodic solution of \eqref{eq:kepleroforzato} which makes $k$ turns in the plane around the origin in each period, where the sign of $k$ gives orientation.
Now, the first fundamental group of $\Lambda_{k}$ is non-trivial when $k\ne 0$: the loop given by $ [0,1]\ni s\mapsto \bar{x}(\cdot+sT)\in\Lambda_{k} $ is not nullhomotopic in $\Lambda_{k}$.
Therefore, one can exploit the family $\mathcal{F}$ of loops in $\Lambda_{k}$ which are not nullhomotopic in order to get the min-max level $c$ of Theorem~\ref{thm:minmax}.

\noindent
Here a problem arises about the class of continuous deformations $\eta$ with respect to which $\mathcal{F}$ is needed to be stable: in the result by Livrea and Marano
the presence of singularities of $\Phi$ is not considered and, thus, the deformations, as usual, are only required to fix the ``boundary''
$(\{0\}\times X)\cup([0,1]\times B)$.
As a consequence, that class of deformations do not comply with the topology of $\Lambda_{k}$: a loop which is not nullhomotopic in $\Lambda_{k}$ could be transformed into a nullhomotopic one by such a deformation $\eta$ that doesn't behave properly near $\partial D_{\Phi}$.
In other words, the family $\mathcal{F}$ we would like to use is not homotopy stable in the usual sense.
On the other hand, we cannot add to $\mathcal{F}$ also all nullhomotopic loops
since, in that case, we could not exclude that the PS-sequence provided by a result like Theorem~\ref{thm:minmax} actually converges to the minimum point $\bar{x}$ itself: we need our loops to really turn around the
singularities of $\Phi$.

\noindent
Our solution consists in adding a condition on the deformations $\eta$ that can be
used in checking the homotopic stability of any set $ A\in\mathcal{F} $:
namely that $\eta([0,1],A)\subset D_{\Phi}$ (see assumption 1 in~\ref{thm:minmax}).
That condition imposes that $\eta(s,A)$ never crosses the boundary of $D_{\Phi}$
and, thus, $\eta(s,A)$ remains not nullhomotopic in $\Lambda_{k}$ as $s$ ranges in $[0,1]$.

\noindent
By reducing the class of admissible deformations $\eta$, we enlarge the
collection of families $\mathcal{F}$ that are homotopy stable and obtain a result which is slightly more general than \cite[Theorem~3.1]{LiMa04}.
Clearly, the main issue is to exploit the fact that the singularities of our functional $\Phi$ satisfy the fourth condition in Assumption~\ref{ass:I}: some small, but non-trivial, adjustments of the proof by Livrea and Marano can accomplish the goal.
Since the whole argument in \cite{LiMa04} is already quite delicate, we decided to provide a complete proof of Theorem~\ref{thm:minmax} in the Appendix.

\begin{Remark}
As well-known, if any PS-sequence possesses a convergent subsequence, the functional $I$ is said to satisfy the Palais-Smale condition. Thus, if a functional $I$ in the assumptions of Theorem \ref{thm:minmax} also satisfies the PS-condition, a critical point of $I$ at level $c$ is directly obtained, in view of the second statement in Proposition \ref{pro:minimoPScritico}. 
In our application, however, we will not make direct use of this strategy. 
Indeed, in spite of the fact that the PS-sequence provided by Theorem \ref{thm:minmax} will be checked to be bounded, due to the specific form of the functional it seems hard to prove that it is relatively compact in the (strong) topology of $X$. Nonetheless, we will manage to prove that, in our setting, the relative compactness in a weaker topology is enough to obtain a critical point at level $c$ (cf. Proposition \ref{pstype}). 
\end{Remark}

\section{The forced problem: multiple solutions}\label{sec3}
In this section we deal with the forced problem
\begin{equation}\label{eq:kepleroforzato2}
\frac{d}{dt}\left(\frac{m\dot{x}}{\sqrt{1-|\dot{x}|^{2}/c^{2}}}\right) =
-\alpha\frac{x}{|x|^{3}} + \nabla_{x}U(t,x), \qquad x \in \mathbb{R}^2 \setminus \{0\},
\end{equation}
and we give the proof of Theorem \ref{thmain}. 
More precisely, in Section \ref{sec3.1} we first reformulate the $T$-periodic problem associated with \eqref{eq:kepleroforzato2}
as a critical point problem for a non-smooth action functional $I$, so as to enter the abstract setting developed in Section \ref{sec2}.
Then, for any integer $k \neq 0$, in Section \ref{sec3.2} we prove the existence of a first $T$-periodic solution to \eqref{eq:kepleroforzato2} with winding number $k$, arising as a local minimum of $I$. Finally, in Section \ref{sec3.3} we prove the existence of the second $T$-periodic solution, again with winding number $k$, by the use of the min-max principle introduced in the previous section.

\noindent
From now on, $m,c,\alpha > 0$ are fixed and the potential $U$ satisfies the following 
\begin{Assumption}\label{ass:U}
	$ U : \RR\times\RR^{2} \to \RR $ is such that:
	\noindent
	\begin{enumerate}
		\item There exists $T>0$ such that $U(t+T,x)=U(t,x), \quad \forall \ (t,x)\in \RR\times\RR^{2}$
		\item $U(t,\cdot)$ is differentiable in $\RR^2$, for every $t\in \RR$
		\item $U$ and $ \nabla_{x}U $ satisfy the $L^{1}$-Carath\'eodory condition, i.e. they are measurable in $t$, continuous in $x$ and for every $r > 0$ there exists $\eta_r \in L^1(0,T)$ such that
		\[
		\vert U(t,x) \vert + \vert \nabla_x U(t,x) \vert \leq \eta_r(t)
		\]
		for a.e. $t \in [0,T]$ and every $x \in \mathbb{R}^2$ with $\vert x \vert \leq r$.
	\end{enumerate}
\end{Assumption}

\subsection{Definition and properties of the action functional}\label{sec3.1}

As already mentioned in the introduction, in what follows we take advantage of some results proved 
by Arcoya, Bereanu and Torres in \cite{ArBeTo20}, where a (non-smooth) variational formulation for Lorentz force equation \eqref{lorentzintro} is provided. In our setting some extra work is needed due to the singularity of the nonlinear term.

\noindent
Let us consider the Banach space
\[
X=W^{1,\infty}_{T}=\left\{ x\in W^{1,\infty}(0,T;\RR^{2}) : x(0)=x(T) \right\},
\]
endowed with its usual norm $ \|x\| = \|x\|_{\infty} + \|\dot{x}\|_{\infty}$. 
We set, for every $s \in \mathbb{R}^2$ with $\vert s \vert\leq c$,
\begin{equation}\label{def-f}
F(s) = mc^2 \left( 1 - \sqrt{ 1 - \frac{\vert s \vert^2}{c^2}}\right),
\end{equation}
and we introduce the functional
$\psi: X \to (-\infty,+\infty]$ defined by
\[
\psi(x) = \begin{dcases}
\int_{0}^{T} F(\dot x(t)) \, dt
& \text{if } \|\dot{x}\|_{\infty} \le c; \\
+\infty & \text{otherwise}.
\end{dcases}
\]
Moreover, we define the functional $ \Phi:X\to(-\infty,+\infty]$ as
$$
\Phi(x) = \int_{0}^{T} \dfrac{\alpha}{|x(t)|} \, dt + \int_{0}^{T} U(t,x(t)) \, dt,\quad \forall \ x\in X.
$$
Of course, $\Phi$ is finite on the open subset
\[
\Lambda = \{ x \in X : x(t) \neq 0 \text{ for all } t \in [0,T] \},
\]
while, due to the singularity of the integrand $\alpha/\vert x \vert$, it could be infinite elsewhere. The simple, but crucial, observation, is that in fact $\Phi(x)$ is always infinite outside $\Lambda$: indeed, if a function $x \in X$ vanishes somewhere in $[0,T]$, due to its Lipschitz continuity the integral
$\int_{0}^{T} \alpha/|x(t)|\, dt$ is surely infinite, and so $\Phi(x) = +\infty$. Since this will be often used in what follows, we state it as a Lemma.

\begin{Lemma}\label{lemfinito}
For every $x \in X$, $\Phi(x) < +\infty$ if and only if $x \in \Lambda$: that is, with the notation of Assumptions~\ref{ass:I}, $D_{\Phi} = \Lambda$.
\end{Lemma} 


\noindent
We then define $I: X \to (-\infty,+\infty]$ as
$$
I(x) = \psi(x) + \Phi(x),\quad \forall \ x\in X.
$$
In the next proposition we show that the functional $I$ is suitable for the abstract setting introduced in Section \ref{sec2}; moreover, we highlight a further important property of the functional $\psi$.

\begin{Proposition}\label{pro:assumptionI}
The functional $ I = \psi + \Phi $ satisfies Assumptions~\ref{ass:I}. Moreover, the functional $\psi$ is lower semicontinuous with respect to uniform convergence, namely: if $x \in X$ and $\{x_n\}$ is a sequence in $D_\psi$ such that $x_n \to x$ uniformly on $[0,T]$, then 
$x \in D_\psi$ and 
\begin{equation}\label{weaklsc}
\psi(x) \leq \liminf_{n \to +\infty} \psi(x_n).
\end{equation}
\end{Proposition}

\begin{proof}
First, we are going to verify separately each point in Assumptions~\ref{ass:I}.
\begin{enumerate}
\item By Lemma \ref{lemfinito} we have $D_{\Phi} = \Lambda$ which is open in $X$. On the other hand,
$D_{\psi} = \{ x \in X \, : \, \Vert \dot x \Vert_\infty \leq c \}$ so that clearly $D_{\psi} \cap D_{\Phi} \neq \emptyset$.
\item The convexity of $\psi$ follows directly from the convexity of the function $F$. 
Moreover, the restriction $\psi|_{D_{\psi}}$ of $\psi$ to $D_{\psi}$ is continuous, as it is easily checked using the dominated convergence theorem (cf. \cite[Lemma 2]{ArBeTo20}). Since $D_{\psi}$ is closed, the lower semicontinuity of $\psi$ on $X$ follows.
\item The functional $\Phi$ is actually of class $C^1$ (and, hence, locally Lipschitz continuous) on the open set $D_{\Phi} = \Lambda$, with
\begin{equation}\label{differenziale}
\Phi'(x)[y] = \Phi^0(x;y) = \int_0^T \left\langle -\alpha \frac{x(t)}{\vert x(t) \vert^3} + \nabla_x U(t,x(t)),y(t) \right \rangle \, dt,
\end{equation}
for any $x \in \Lambda$ and $y \in X$ (see, for instance, \cite{MaWi89}).
\item At first, we notice that 
$$
\partial D_{\Phi} = X \setminus \Lambda = \{x \in X : x(t) = 0 \textrm{ for some } t \in [0,T]\}.
$$
So, let us consider a sequence $\{x_n\}$ in $D_{\Phi} \cap D_{\psi}$ such that $d_n := \dist(x_n, \partial D_{\Phi}) \to 0$
and, accordingly, let $y_n \in \partial D_{\Phi}$ be such that $\Vert x_n - y_n \Vert \leq 2d_n$.
Since $\Vert \dot x_n \Vert_\infty \leq c$ for any $n$, we find that
$$
\Vert \dot y_n \Vert_\infty \leq c + \Vert \dot y_n - \dot x_n  \Vert_\infty \leq c + \Vert  y_n - x_n  \Vert \leq c + d_n \leq c+1
$$
for $n$ large enough. Moreover, since $y_n(t_n) = 0$ for some $t_n \in [0,T]$, we have
$$
\Vert y_n \Vert_\infty \leq (c+1)T
$$
and thus the sequence $\{y_n\}$ is bounded in $X$. Since $\Vert x_n - y_n \Vert \leq 2d_n$, the sequence $\{x_n\}$ is bounded in $X$ as well. Therefore, the Ascoli-Arzelà theorem yields the existence of a continuous function $z$ such that, up to subsequence, $x_n \to z$ and $y_n \to z$ uniformly on $[0,T]$. Hence, $z(0) = z(T)$ and $z(\bar{t}) = 0$ for some $\bar{t} \in [0,T]$, limit point of the sequence $t_n$. Moreover, passing to the limit in the Lipschitz-continuity condition
$$
\vert x_n(t_2) - x_n(t_1) \vert \leq c |t_2-t_1|, \quad \text{ for every } t_1,t_2 \in [0,T],
$$
we easily see that $z \in D_\psi \subset X$ and so, by Lemma \ref{lemfinito},
$$
\int_{0}^{T} \dfrac{\alpha}{|z(t)|}\, dt = +\infty.
$$
Hence, by Fatou's lemma
$$
\liminf_{n \to +\infty} \int_0^T \frac{\alpha}{\vert x_n(t) \vert}\, dt  \geq \int_0^T \liminf_{n \to +\infty} \frac{\alpha}{\vert x_n(t) \vert}\, dt  = 
\int_{0}^{T} \dfrac{\alpha}{|z(t)|}\, dt = +\infty.
$$
Since $\psi(x_{n})$ and $\int_0^T U(t,x_n(t))dt$ remain bounded by the uniform convergence of $x_n$ to $z$, we finally conclude that $I(x_n) \to +\infty$ as desired.
\end{enumerate}
Finally, for the proof of the lower semincontinuity property \eqref{weaklsc} we refer the reader to \cite[Lemma 3]{ArBeTo20} (dealing with Dirichlet boundary condition; the periodic problem is taken into account in Section 4 of the same paper).
\end{proof}

\noindent
The following proposition establishes a compactness property of the functional $I$. Notice, however, that 
what we are going to prove differs from the usual Palais-Smale condition: indeed, the convergence of a subsequence in a weaker topology is obtained.

\begin{Proposition}\label{pstype}
Let $\{x_n\} \subset X$ be a bounded Palais-Smale sequence at level $c$. Then, up to subsequence, $x_n \to x$ uniformly on $[0,T]$, with $x$ a critical point of $I$ satisfying $I(x) = c$. 
\end{Proposition}

\begin{proof}
We follow closely the arguments in \cite[Lemma 5]{ArBeTo20}, with just some more care due to the presence of the singularity.
Let $\{x_n\} \subset X$ be a bounded Palais-Smale sequence at level $c$ and notice that this implies $\{x_n\} \subset D_{\Phi} \cap D_{\psi}$ (otherwise $I(x_n) = +\infty$). Using the Ascoli-Arzelà theorem as in the proof of the fourth point of Proposition \ref{pro:assumptionI}, we thus infer the existence of $x \in D_\psi$ such that, up to subsequence, $x_n \to x$ uniformly on $[0,T]$; moreover, by \eqref{weaklsc},
\begin{equation}\label{scancora}
\psi(x) \leq \liminf_{n \to +\infty} \psi(x_n).
\end{equation}
On the other hand, Assumption~\ref{ass:I}-4 yields that $x \notin \partial D_{\Phi}$, since otherwise it would be
$I(x_n) = \psi(x_n) + \Phi(x_n) \to +\infty$, contrarily to $I(x_n) \to c$.
Thus, recalling that $\partial D_{\Phi} = X \setminus \Lambda$, we conclude that $x \in \Lambda$ and, hence,
\begin{equation}\label{convphi}
\Phi(x) = \lim_{n \to +\infty} \Phi(x_n) \quad \text{ and } \quad \Phi'(x)[z-x] = \lim_{n \to +\infty} \Phi'(x_n)[z-x_n],\quad \forall \ z \in X,
\end{equation}
where the last equality is easily obtained by using the explicit formula \eqref{differenziale}. By the above relations together with \eqref{scancora}, passing to the limsup in the inequality characterizing the PS-sequence we thus found
$$
\psi(z) - \psi(x) + \Phi'(x)[z-x] \geq 0, \quad \forall z \in X,
$$
proving that $x$ is a critical point of $I$. It remains to show that $I(x) = c$. For this, we choose $z = x$ in the inequality of the PS-sequence so as to obtain
$$
\psi(x_n) \leq \psi(x) + \Phi'(x_n)[x-x_n] + \epsilon_n \Vert x - x_n \Vert .
$$
Thus, $\limsup_{n \to +\infty} \psi(x_n) \leq \psi(x)$ which together with \eqref{scancora} yields
$\psi(x) = \lim_{n \to +\infty} \psi(x_n)$. 
Recalling \eqref{convphi}, we finally find 
$$
I(x) = \lim_{n \to +\infty} I(x_n) = c,
$$ 
thus concluding the proof.
\end{proof}

\noindent
The rest of the section is devoted to the proof that critical points of the functional $I$ correspond to $T$-periodic solutions of 
\eqref{eq:kepleroforzato2}, as stated in the proposition below. 

\begin{Proposition}\label{pro:criticisoluzioni}
Any critical point of $I$, according to Definition~\ref{def:puntocritico}.1, is a $T$-periodic 
solution of \eqref{eq:kepleroforzato2}.
\end{Proposition}

\noindent
The above statement essentially corresponds to \cite[Theorem 6]{ArBeTo20}. We point out that the difficulties in establishing this result come from the fact that, a priori, a critical point $x \in D_I$ could satisfy $\Vert \dot x \Vert_\infty = c$: to exclude this, an ad-hoc argument is thus required. On the other hand, the singularity of the nonlinear term does not play a real role, since a critical point satisfy $x \in D_\Phi$ by definition and the notion of critical point is of local nature. However, since the complete argument is quite delicate, we chose here to provide the complete proof. As the one in \cite{ArBeTo20}, it relies in an essential way on some preliminary lemmas established in \cite{BeJeMa11,BeMa08} for the auxiliary problem
\begin{equation}\label{eq:phifT}
\begin{dcases}
\frac{d}{dt}\left(\frac{m\dot{u}}{\sqrt{1-|\dot{u}|^{2}/c^{2}}}\right) = f(t) & \text{in }[0,T]\\
u(0)-u(T)=0=\dot{u}(0)-\dot{u}(T) &
\end{dcases}
\end{equation}
with $f\in L^{1}(0,T) $, and the related functional $ J:X\to\left(-\infty,+\infty\right] $
defined by
\begin{equation}\label{eq:J}
J(u) = \psi(u) + \int_{0}^{T} \langle f(t), u \rangle \, dt,\quad \forall \ u\in X.
\end{equation}
More precisely, we are going to make use of the next result.

\begin{Lemma}\label{lem:aux}
The following statements hold:
\begin{enumerate}
\item
problem~\eqref{eq:phifT} has a solution if and only if $\int_{0}^{T} f(t)\, dt =0$;
in this case, the set of solutions is given by $ \{ u_{0} + c \} $ where $ u_{0} $ is a fixed
solution and $c$ is any constant vector;
\item
the functional $J$ has a global minimum if and only if $\int_{0}^{T} f(t)\, dt =0$;
in this case, the set of minima  is given by $ \{ u_{0} + c \} $ where $ u_{0} $ is a minimum and $c$ is any constant vector;
\item
$ u $ is a solution of \eqref{eq:phifT} if and only if it is a global minimum of $J$.
\end{enumerate}
\end{Lemma}
\begin{proof}
It is convenient to introduce the following notation. Define, for $s \in \mathbb{R}^2$ with $\vert s \vert < c$,
$$
\varphi(s) = \frac{m s}{\sqrt{1-|s|^{2}/c^{2}}},
$$
so that the equation in \eqref{eq:phifT} becomes $\frac{d}{dt} (\varphi(\dot u )) = f(t)$.
It is readily verified that $\varphi$ is a global homeomorhism of $B_c(0) \subset \mathbb{R}^2$ onto $\mathbb{R}^2$ and, moreover, 
$$
\nabla F(s) = \varphi(s), \quad \forall \ s \in B_c(0),
$$
where $F$ is defined in \eqref{def-f}.
We are now in a position to prove the various parts of the statement.
\begin{enumerate}
\item Integrating the equation and using the fact that $\varphi$ is a global homeomorhism of $B_c(0)$ onto $\mathbb{R}^2$, we find
$$
\dot u(t) = \varphi^{-1} \left( \varphi(\dot u(0)) + \int_0^t f(s)\, ds \right), \quad \forall \ t \in [0,T].
$$
Hence, since it has to be $\dot u(0) = \dot u(T)$, the condition $\int_0^T f(t)\, dt = 0$ is easily seen to be necessary for the solvability of problem \eqref{eq:phifT}.

\noindent
On the contrary, suppose that such a condition is satisfied. A further integration yields
$$
u(t) = u(0) + \int_0^t \varphi^{-1} \left( \varphi(\dot u(0)) + \int_0^s f(\tau)\, d\tau \right) ds, \quad \forall \ t \in [0,T].
$$
Therefore, $u$ is a solution of \eqref{eq:phifT} if and only if $u(0) = u(T)$, that is
$$
\int_0^T \varphi^{-1} \left( \varphi(\dot u(0)) + \int_0^s f(\tau)d\tau \right)\, ds = 0,
$$
and, in such a case, the function $\tilde u = u + c$ is still a solution for every $c \in \mathbb{R}^2$. To conclude the proof, we thus need to show that, for any fixed $f \in L^1(0,T)$ with $\int_0^T f = 0$, the equation
$$
\int_0^T \varphi^{-1} \left( d + \int_0^s f(\tau)\, d\tau \right) ds = 0, \quad d \in \mathbb{R}^2,
$$
admits a unique solution: this is proved in \cite[Lemma 2]{BeMa08} (see also \cite[Example 2]{BeMa08}).
\item We first notice that, if $\int_0^T f(t)\, dt \neq 0$, then
$$
J(c) = \left\langle \int_0^T f(t)\, dt, c \right\rangle 
$$
for any constant function $u(t) \equiv c \in \mathbb{R}^2$; hence, the functional $J$ is unbounded below (and above). 
Then, $\int_0^T f(t)\, dt =0$ is a necessary condition for the existence of global minima of $J$. 

\noindent
Second, we prove that if $u_1,u_2$ are both global minima of $J$, then $\dot u_1 = \dot u_2$ a.e. on $[0,T]$. Indeed, if $\dot u_1 \neq \dot u_2$ on a positive measure set $S \subset [0,T]$, then from the strict convexity of $F$ it holds
$$
F(\lambda \dot u_1(t) + (1-\lambda) \dot u_2(t)) \leq \lambda F(\dot u_1(t)) + (1-\lambda) F(\dot u_2(t)), \quad \text{ a.e. in } [0,T],
$$
with strict inequality when $t \in S$. Then
$$
J(\lambda u_1 + (1-\lambda) u_2) < \lambda J(u_1) + (1-\lambda) J(u_2),
$$
contradicting the fact that $u_1,u_2$ are global minima. From $\dot u_1 = \dot u_2$ a.e, it follows of course that $u_2 - u_1 \equiv c$ for some $c \in \mathbb{R}^2$, proving the desired characterization of the set of global minima. 

\noindent
It thus remains to prove the existence of global minima when the condition $\int_0^T f(t)\, dt =0$ is satisfied. For this, we refer to the next point 3, showing in fact that solutions of \eqref{eq:phifT} (which exist, by point 1, in such a case) are global minima of $J$

\item Let $u$ be a solution of \eqref{eq:phifT} and take $v \in X$. We can assume $\Vert \dot v \Vert_\infty \leq c$, otherwise the conclusion is obvious; recalling that $\nabla F(s) = \varphi(s)$, the convexity of $F$ gives
$$
F(\dot v(t)) \geq F(\dot u(t)) + \langle \varphi(\dot u(t)), \dot v(t) - \dot u(t)\rangle, \quad \text{ a.e. in } [0,T].
$$
Then, using the above equality together with the fact that $u$ solves \eqref{eq:phifT}, we obtain
\begin{align*}
J(v) - J(u) &= \int_0^T \big( F(\dot v(t)) - F(\dot u(t)) + \langle f(t), v(t) - u(t) \rangle \big)\, dt \\
& \geq \int_0^T \big(  \langle \varphi(\dot u(t)), \dot v(t) - \dot u(t) \rangle + \langle f(t), v(t) - u(t) \rangle\big)\, dt \\
& =  \int_0^T \left(  -\left\langle \frac{d}{dt} \left( \varphi(\dot u(t))\right), v(t)-u(t) \right\rangle + \langle f(t), v(t) - u(t) \rangle\right)\, dt = 0,
\end{align*}
proving that $u$ is a global minimum of $J$.

\noindent
The fact that any global minimum of $J$ is a solution of \eqref{eq:phifT} is now an obvious consequence of what we have proved so far.
Indeed, let us take $\bar{u}$ solution of \eqref{eq:phifT} and, thus, global minimum of $J$. Since any other global minimum $u$ of $J$ is of the form $u = \bar{u} + c$ for some $c \in \mathbb{R}^2$, we immediately see that $u$ is a solution of \eqref{eq:phifT}, as well.
\end{enumerate}

\end{proof}

\noindent
We are now in a position to provide the proof of Proposition \ref{pro:criticisoluzioni}.

\begin{proof}[Proof of Proposition \ref{pro:criticisoluzioni}]
Let $ x\in D_{I} $ be a critical point of $I$ and define
\[
f(t) = -\alpha\frac{x(t)}{|x(t)|^{3}} + \nabla_{x}U(t,x(t)),\quad \forall \ t\in [0,T].
\]
Since $\nabla_{x}U$ is $L^1$-Carath\'eodory, we deduce that $f\in L^{1}(0,T)$.

\noindent
Recalling \eqref{differenziale}, the fact that $x$ is a critical point of $I$ reads as 
$$
\int_0^T \langle f(t),u(t) - x(t) \rangle \, dt + \psi(u) - \psi(x) \geq 0, \quad \forall \ u \in X,
$$
which means that $x$ is a global minimum of the functional $J$ defined in \eqref{eq:J}.

\noindent
Hence, Lemma~\ref{lem:aux} applies and the proposition is proved.
\end{proof}

\begin{Remark}
Let us mention that the converse of Proposition \ref{pro:criticisoluzioni} also holds true: that is, $T$-periodic 
solutions of \eqref{eq:kepleroforzato2} are critical points of the functional $I$. This however can be proved with standard arguments, since the functional $I$ is smooth on the interior of its domain $D_I$.
\end{Remark}

\subsubsection{Homotopy classes of paths} \label{subsec:hom}

Since we are interested in finding periodic solutions with prescribed winding number, for any integer $k$ we define the set
\begin{equation}\label{def:homclass}
\Lambda_k = \left\{ x \in \Lambda : \textrm{i}(x) = k\right\},
\end{equation}
where $\textrm{i}(x)$ is the winding number of $x = (x_u,x_v)$ as a closed path in $\mathbb{R}^2 \setminus \{0\}$, that is
$$
\textrm{i}(x) = \frac{1}{2\pi} \int_x \frac{u dv - v du}{u^2 + v^2} =  \frac{1}{2\pi} \int_0^T \frac{x_u(t) \dot x_v(t)  - x_v(t) \dot x_u(t)}{x_u(t)^2 + x_v(t)^2} \,dt. 
$$
Some useful properties of the set $\Lambda_k$ are collected in the next lemma.
\begin{Lemma}\label{proplambdak}
The following hold true:
\begin{itemize}
\item[(1)] $\Lambda_k$ is open with respect to the topology of uniform convegence, and so, a fortiori, $\Lambda_k$ is an open subset of $X$; in particular, for any $x \in \Lambda_k$, it holds that
\begin{equation}\label{rouche}
\{ y \in X : \Vert y - x \Vert_\infty < \min_t \vert x(t) \vert\} \subset \Lambda_k.
\end{equation}
\item[(2)] $\Lambda_k \cap D_\psi \neq \emptyset$ and, for $k \neq 0$,  
\begin{equation}\label{stima_contr}
\Vert x \Vert_\infty \leq cT, \quad \text{ for every } x \in \Lambda_k \cap D_\psi.
\end{equation}
\end{itemize}
\end{Lemma}
\begin{proof}
(1) The assertion, as well as the inclusion \eqref{rouche} (coming from the so-called Rouché's property of the winding number), is well known. \medskip \\
(2) It is easily checked that the path 
$$
\bar{x}(t) := \rho e^{\frac{2\pi k t}{T}i}, \qquad t \in [0,T],
$$
belongs to $\Lambda_k \cap D_\psi$ as long as $0 < \rho \leq cT /(2\pi k )$, showing that $\Lambda_k \cap D_\psi \neq \emptyset$.

\noindent 
To prove \eqref{stima_contr}, we suppose on the contrary that $|x(\bar{t})| > cT$ for some $\bar{t} \in [0,T]$.
Then, for every $t \in [0,T]$ we find
$$
|x(t) - x(\bar{t}) | \leq \left\vert \int_{t}^{\bar{t}} \dot x(s)\, ds \right\vert \leq cT,
$$
so that $x([0,T]) \subset \overline{B}_{cT}(x(\bar{t})) \subset \mathbb{R}^2 \setminus \{0\}$.
Since the ball is conctractible, we have thus reached a contradiction with the fact that $x \in \Lambda_k$ with $k \neq 0$.
\end{proof}

\subsection{The first solution: minimisation}\label{sec3.2}
Let us fix $k \neq 0$. We are going to prove that the minimization problem
\begin{equation}\label{def_min}
\min_{x \in \Lambda_k \cap D_\psi} I(x)
\end{equation}
has a solution, that is, there exists $\bar{x} \in \Lambda_k \cap D_\psi$ such that
$$
I(\bar{x}) \leq I(x), \quad \text{ for every } x \in \Lambda_k \cap D_\psi.
$$
Since, by Lemma \ref{proplambdak} the set $\Lambda_k$ is open in $X$ and $I(x) = +\infty$ if $x \notin D_\psi$, we easily see that $\bar{x}$ is a local minimum of $I$. Then, by Proposition \ref{pro:minimoPScritico} $\bar{x}$ is a critical point of $I$ and Proposition \ref{pro:criticisoluzioni} finally implies that $\bar{x}$ is a $T$-periodic solution to \eqref{eq:kepleroforzato2}, with winding number equal to $k$.

\noindent
So, let us define
$$
m_k = \inf_{x \in \Lambda_k \cap D_\psi} I(x)
$$ 
and consider a minimizing sequence $\{ x_n \} \subset \Lambda_{k} \cap D_{\psi}$, that is, $I(x_n) \to m_k$.
Since $\Vert x_n \Vert_\infty \leq cT$ by estimate \eqref{stima_contr} and, moreover, $\Vert \dot x_n \Vert_\infty \leq c$, 
the sequence $\{x_n\}$ is bounded in $X$. Hence, using the Ascoli-Arzelà theorem as in the proof of the fourth point of Proposition \ref{pro:assumptionI}, we infer the existence of $\bar{x} \in D_\psi$ such that, up to subsequence, $x_n \to \bar{x}$ uniformly on $[0,T]$; moreover, by \eqref{weaklsc},
\begin{equation}\label{stima1}
\psi(\bar{x}) \leq \liminf_{n \to +\infty}\psi(x_n).
\end{equation}
Similarly as in the proof of Proposition \ref{pstype}, we have that $\bar{x} \notin \partial D_{\Phi}$, since otherwise, by  Assumption~\ref{ass:I}-4, it would be
$I(x_n) = \psi(x_n) + \Phi(x_n) \to +\infty$ contrarily to $I(x_n) \to m_k$.
Thus, recalling that $\partial D_{\Phi} = X \setminus \Lambda$, we conclude that $\bar{x} \in \Lambda$ and, by \eqref{stima1} together with the smoothness of $\Phi$ on $\Lambda$, $$
I(\bar{x}) \leq \liminf_{n \to +\infty} I(x_n) = m_k.
$$
Since $x_n \to \bar{x}$ uniformly and $\Lambda_k$ is open w.r.t. the topology of uniform convergence, $\bar{x} \in \Lambda_k$, as well. Then, $\bar{x} \in \Lambda_k \cap D_\psi$ is a solution of the minimization problem \eqref{def_min}, as desired.

\subsection{The second solution: a min-max argument}\label{sec3.3}
Let $\bar{x} \in \Lambda_k \cap D_\psi$, for $k\neq 0$, be a solution of the minimization problem \eqref{def_min} considered in the previous section. We are going to apply the min-max principle given in Theorem \ref{thm:minmax}. To this end, we define
$$
B = \{\bar x \}
$$
and
$$
\mathcal{F} = \{ \gamma([0,1]) : \gamma \in C(\mathbb{T}^1,\Lambda_k),\; \gamma \textrm{ is not nullhomotopic and } \gamma(0) = \bar{x}\},
$$
that is, the set of the images of all non nullhomotopic loops based at $\bar{x}$ in the topological space $\Lambda_k$ (here $\mathbb{T}^1$ is meant as the
quotient space $\mathbb{R}/\mathbb{Z}$, so that $\gamma(0) = \gamma(1) = \bar{x}$). Of course $B$ is closed in $X$ and $\mathcal{F}$ is a family of compact sets of $X$. Notice that $\mathcal{F}$ is non empty since the compact set
\begin{equation}\label{defbara}
\bar{A} :=\bar{\gamma}([0,1]), \quad \textrm{ with } \bar{\gamma}(s) = \bar{x}(T s+\cdot) \in \Lambda_k, 
\end{equation}
belongs to $\mathcal{F}$. 
We now verify the other assumptions of Theorem \ref{thm:minmax}.

\noindent
First, we prove that $\mathcal{F}$ is homotopy stable with extended boundary $B$. So, let us consider $ A = \gamma([0,1]) \in \mathcal{F} $ and a continuous function $ \eta: [0,1]\times X \to X$ such that 
$$ 
\eta(t,x)=x, \quad \forall (t,x)\in (\{0\}\times X)\cup ([0,1]\times B)  \qquad \text{ and } \qquad
\eta([0,1]\times A)\subset D_{\Phi}.
$$ 
Preminarily, we observe that, for any $s \in [0,1]$, 
the function $t \mapsto \eta(t,\gamma(s)) \in D_{\Phi}$ provides an homotopy of the path $\gamma(s)$ (which is obtained at $t = 0$) in the punctured plane $\RR^{2}\setminus\{0\}$: since $\gamma(s) \in \Lambda_k$, by the homotopy invariance of the winding number it has to be $\eta(t,\gamma(s)) \in \Lambda_k$ for any $t \in [0,1]$. Now, let us set $\gamma_t(s) = \eta(t,\gamma(s))$ for every $(t,s) \in [0,1]\times[0,1]$. It is easily checked that, for every $t \in [0,1]$,  
$\gamma_t$ is a loop of base $\bar{x}$ in the topological space $\Lambda_k$. Moreover, by construction all these loops are homotopic:
since $\gamma_0 = \gamma$ is not nullhomotopic, we deduce that $\gamma_1$ is not nullhomotopic as well. 
We have thus proved that $\eta(1,A) = \eta(1,\gamma([0,1]) \in \mathcal{F}$, as desired.

\noindent
Second, recalling the definition of $\bar{A} \in \mathcal{F}$ given above, we have
$$
c:= \inf_{A \in \mathcal{F}}\sup_{x \in A} I(x) \leq \sup_{x \in \bar{A}} I(x) = \sup_{s \in [0,1]} I(\bar{x}(sT + \cdot)) < +\infty
$$
since $\psi(\bar{x}(sT + \cdot)) = \psi(\bar{x})$ and $\Phi(\bar{x}(sT + \cdot))$ is bounded, due to the fact that the functions
$\bar{x}(sT + \cdot)$ are uniformly bounded and uniformly bounded away from the singularity.

\noindent
Finally, we define
$$
F = \{x \in X : \Vert x - \bar{x} \Vert_\infty = r\},
$$
with $r \in (0,\min_t |\bar{x}(t)|)$ arbitrarily chosen. We stress that, in the above formula, the sup norm $\Vert \cdot \Vert_\infty$ is considered, even if the ambient space is $X = W^{1,\infty}_T$; since uniform convergence is implied by convergence in $X$, the set $F$ is closed in $X$ Let us also notice that, by the Rouché's property \eqref{rouche}, the set $\{x \in X :  \Vert x - \bar{x} \Vert_{\infty} \leq r\}$ is entirely contained in $\Lambda_k$; hence we have
\[
\sup_B I = I(\bar{x}) \leq \inf_F I.
\]
We claim that $A \cap F \neq \emptyset$ for every $A = \gamma([0,1]) \in \mathcal{F}$. Indeed,
if otherwise $\Vert \gamma(s) - \bar{x} \Vert_{\infty} < r$ for every $s \in [0,1]$, the convex deformation
$\nu(\lambda,s) = \lambda \bar{x} + (1-\lambda) \gamma(s)$ (with $\lambda \in [0,1]$) would imply that $\gamma$ is homotopic in $\Lambda_k$ to the constant loop of base $\bar{x}$, against the assumption.
Since $F \cap B = \emptyset$, the condition $(A \cap F) \setminus B\neq\emptyset$ of Theorem \ref{thm:minmax} thus follows.

\noindent
Summing up, we can apply Theorem \ref{thm:minmax}. Thus, fixed a sequence $\{A_n \}$ in $\mathcal{F}$ such that
$\sup_{A_n} I \to c$ we can find a PS-sequence $\{x_n \}$ in $X$, at level $c$, such that
$\dist(x_n,A_n) \to 0$; moreover, if $\inf_F I = c$, then $\dist(x_n, F) \to 0$. Now, let us recall that $A_n \subset \Lambda_k$ and observe that, since $I(x) = +\infty$ if $x \notin D_{\psi}$ and $\sup_{A_n} I \to c$, it must be $A_n \subset D_{\psi}$, as well. Hence, $A_n \subset D_{\Phi} \cap D_{\psi}$, so that the fourth condition in Assumption~\ref{ass:I} implies that 
$\dist(A_n,\partial \Lambda_k)$ remains bounded away from zero.
Hence, from $\dist(x_n,A_n) \to 0$ we can infer that $x_n \in \Lambda_k$ if $n$ is large enough.
Using the inequality \eqref{stima_contr}, we thus find that $\{x_n\}$ is bounded.
Then, Proposition \ref{pstype} applies, providing the existence of a critical point $x^*$ of $I$ at level $c$, obtained as the uniform limit of a subsequence of $\{x_n\}$.
By the first assertion in Lemma \ref{proplambdak}, $x^* \in \Lambda_k$.

\noindent
If $\inf_F I < c$, since $I(\bar{x}) \leq \inf_F I \leq c = I(x^*)$ by assumptions 2 and 3 of Theorem~\ref{thm:minmax} (see its proof in the appendix), we get $x^* \neq \bar{x}$. If, otherwise, $\inf_F I = c$, then we also know that $\dist(x_n, F) \to 0$. Then, since 
$\dist_{\infty}(x_n,F) := \inf_{y \in F} \Vert x_n - y \Vert_\infty \leq \dist(x_n, F)$, we have 
$\dist_{\infty}(x_n,F) \to 0$.
On the other hand, $\dist_{\infty}(x_n,F) \to \dist_{\infty}(x^*,F)$ and so $\dist_{\infty}(x^*,F) = 0$.
Hence
$$
r = \dist_{\infty}(\bar{x},F) \leq \Vert \bar{x} - x^* \Vert_\infty + \dist_{\infty}(x^*,F) = \Vert \bar{x} - x^* \Vert_\infty,
$$
implying $x^* \neq \bar{x}$. Hence, in both the cases, a second $T$-periodic solution to \eqref{eq:kepleroforzato2}, with winding number equal to $k$, is found.

\begin{Remark}\label{rem:continuo}
We observe that the equality $I(\bar{x}) = c$ surely holds for the autonomous equation
$$
\frac{d}{dt}\left(\frac{m\dot{x}}{\sqrt{1-|\dot{x}|^{2}/c^{2}}}\right) =
-\alpha\frac{x}{|x|^{3}} + \nabla U(x).
$$
Indeed, in this case the functional $I$ is invariant under time-translation and hence, recalling the definition in \eqref{defbara}, 
$$
c= \inf_{A \in \mathcal{F}}\sup_{x \in A} I(x) \leq \sup_{x \in \bar{A}} I(x) = I(\bar{x}).
$$
In this situation the functional $I$ possesses the continuum of local minima $\bar{A}$.
\end{Remark}
\begin{Remark}\label{Ncentri}
With minor modifications of the arguments used in this section, Theorem \ref{thmain} could be extended to more general equations. For instance, it is possible to obtain a similar result for a forced relativistic $N$-center problem of the type 
\begin{equation}\label{eq:ncentri}
\frac{d}{dt}\left(\frac{m\dot{x}}{\sqrt{1-|\dot{x}|^{2}/c^{2}}}\right) =
-\sum_{j=1}^N \alpha_j \frac{x-\xi_j}{|x-\xi_j|^{3}} + \nabla_{x}U(t,x), \qquad x \in \RR^{2}\setminus\{\xi_1, \ldots, \xi_N\},
\end{equation}
where $\xi_1,\ldots \xi_N$ are fixed points in $\RR^2$ and $\alpha_j>0$, for every $j=1,\ldots, N$ (clearly, equation \eqref{eq:ncentri} reduces to \eqref{eq:kepleroforzato2} for $N=1$ and $\xi_1=0$). In this more general setting, the existence of $T$-periodic solutions can be proved in any non-trivial homotopy class $\Gamma \in \pi_{1}(\RR^{2}\setminus\{\xi_{1},\dots,\xi_{n}\})$ such that $\Gamma \cap D_\psi \neq \emptyset$. Notice that, in the case of $N \geq 2$ centers, this requirement is essential: indeed, the natural bound $|\dot{x}|\le c$ for any $x\in D_{\psi}$ implies that 
\[
\max_{t,s \in [0,T]} |x(t)-x(s)|\leq cT,
\]
and, thus, solutions $x$ to problem \eqref{eq:ncentri} cannot belong to a class $\Gamma$ containing loops with
non-zero winding number around $\xi_{i}$ and $\xi_{j}$, if $|\xi_{i}-\xi_{j}|>cT$.
\end{Remark}
\section{The unforced problem: minimal solutions}\label{sec4}
In this section we deal with the unforced problem 
\begin{equation}\label{eq:keplero}
\frac{d}{dt}\left(\frac{m\dot{x}}{\sqrt{1-|\dot{x}|^{2}/c^{2}}}\right) =
-\alpha\frac{x}{|x|^{3}}, \qquad x \in \RR^{2}\setminus\{0\},
\end{equation}
with the final goal of giving the proof of Theorem \ref{thmain2}.
More precisely, in Section \ref{subsec:solexistence} we first give a complete description of the set of circular and non-circular $T$-periodic solutions of \eqref{eq:keplero} with winding number $k$, for an arbitrary $T > 0$ and an arbitrary integer $k \neq 0$. Then, in Section \ref{subsec:rosette} and Section \ref{subsec:circ} we provide explicit formulas for the action level of non-circular solutions and for the Morse index of the circular solutions, respectively. Finally, using these preliminary results, in Section \ref{subsec:minimalita} we give a more precise statement of Theorem \ref{thmain2} and we provide its proof.

\noindent
Due to the invariances of \eqref{eq:keplero}, in what follows the uniqueness of a periodic solution
will always be meant up to time translations and space rotations around the origin.
Moreover, without loss of generality we consider only solutions having positive winding number $k$ around the origin (equivalently, with positive angular momentum, see \eqref{eq:thetaprimo}), that is, solutions belonging to the set $\Lambda_k$ defined in \eqref{def:homclass}, with $k \geq 1$.

\subsection{Circular and non-circular periodic solutions}\label{subsec:solexistence}

As we will see, problem \eqref{eq:keplero} can have circular or non-circular periodic solutions. 
As far as circular solutions are concerned, we notice that a circular $T$-periodic solution belongs to $\Lambda_k$ if and only if its minimal period is $T/k$. With this in mind, it is immediate to prove the following result.
\begin{Proposition} \label{pro:circolari}
	For every $T>0$ and $k\in \NN$, with $k\geq 1$, problem \eqref{eq:keplero} possesses a unique circular $T$-periodic solution $x_C\in \Lambda_k$, given by
	\begin{equation} \label{eq:solcirc}
	x_C(t)=R\, e^{i\omega t},\quad R=\dfrac{L\sqrt{L^2c^2-\alpha^2}}{m\alpha c},\ \omega=\dfrac{m\alpha^2c}{L^2\sqrt{L^2c^2-\alpha^2}},
	\end{equation}
	where (the angular momentum) $L>0$ is the solution of 
	\begin{equation} \label{eq:valoreL}
	L^2 \sqrt{L^2c^2-\alpha^2}=\dfrac{m\alpha^2 c}{2\pi k}\, T.
	\end{equation}
\end{Proposition}

\noindent
On the other hand, the existence of non-circular periodic solutions of \eqref{eq:keplero} can be proved by means of a phase-plane analysis, on the lines of the recent paper \cite{BoDaFe22}, which presents a precise description of the solutions, for a given energy and angular momentum.

\noindent
It is important to observe, as it will be clear from the proof of Proposition~\ref{pro:rosette}, that problem \eqref{eq:keplero} does not have non-circular periodic solutions in $\Lambda_1$. Moreover, for $k\geq 2$, non-circular periodic solutions in $\Lambda_k$ can be classified according to the following definition.
\begin{Definition} \label{def:rosette}
	Let $T>0$, $k\in \NN$, with $k\geq 2,$ and $n\in \{1,\ldots, k-1\}$. A non-circular $T$-periodic solution $x$ of \eqref{eq:keplero} is said to be a solution of type $(n,k)$ if $x\in \Lambda_k$ and $|x|$ has minimal period $T/n$. 
\end{Definition}
\noindent
Let us define
\begin{equation} \label{eq:tstar}
T_*= \dfrac{2\pi \alpha}{mc^3}
\end{equation}
and, for every $k\in \NN,\ k\geq 1$
\begin{equation} \label{eq:ukn}
u^k_n=
\begin{cases}
\dfrac{nk^3}{(k^2-n^2)^{3/2}}\, T_*& \mbox{if $n\in \{0,\ldots, k-1\}$}
\vspace{4pt} \\
+\infty & \mbox{if $n=k$.}
\end{cases}
\end{equation}
We are now in position to state the result on the existence of non-circular periodic solutions of \eqref{eq:keplero}.
\begin{Proposition} \label{pro:rosette}
	Let $T>0$, $k\in \NN$, with $k\geq 2,$ and $i_T\in \{0,\ldots, k-1\}$ such that
	\begin{equation} \label{eq:condTrosetta}
	u^k_{i_T}< T\leq  u^{k}_{i_T+1}.
	\end{equation}
	Then:
	\begin{enumerate}
		\item if $i_T=0$ problem \eqref{eq:keplero} does not admit non-circular $T$-periodic solutions in $\Lambda_k$
		\item if $i_T>0$ problem \eqref{eq:keplero} has exactly $i_T$ non-circular $T$-periodic solutions in $\Lambda_k$ and they are of type $(1,k), \ldots, (i_T,k)$
	\end{enumerate}
\end{Proposition}
\begin{proof}
	The proof is based on the phase-plane analysis given in \cite[Sect. 2]{BoDaFe22}. In that paper the existence of periodic solutions of \eqref{eq:keplero} is studied depending on the values of the energy
	\begin{equation} \label{eq:energia}
	h=\dfrac{mc^2}{\sqrt{1-|\dot{x}|^2/c^2}}-\dfrac{\alpha}{|x|}
	\end{equation}
	and of the angular momentum
	\begin{equation} \label{eq:momento}
	L = \langle x, Jp \rangle, \qquad \text{where } J = \begin{pmatrix} 0 & 1 \\
		-1 & 0\end{pmatrix},\ p=\dfrac{m\dot{x}}{\sqrt{1-|\dot{x}|^2/c^2}}.
		\end{equation}
More precisely, it is shown that non-circular periodic solutions exist if and only if $(h,L)$ belongs to the set
\[
\Sigma=\left\{(h,L)\in \RR^2:\ 0<h<mc^2,\ \dfrac{\alpha^2}{c^2}<L^2 < \dfrac{\alpha^2 m^2 c^2}{m^2c^4-h^2} \right\}.
\]
Moreover, when $(h,L)\in \Sigma$, the radial component $|x|$ of a solution $x$ of \eqref{eq:keplero} is always periodic and its period $T_h$ depends on the energy $h$ according to the relation
\begin{equation} \label{eq:periodorad}
T_h=\dfrac{2\pi \alpha m^2 c^3}{(m^2c^4-h^2)^{\frac{3}{2}}}
\end{equation}
(see \cite[Prop. 2.1]{BoDaFe22}).
On the other hand, the angular component $\theta$ of $x$ has a total variation in the interval $[0,T_h]$ given by
\begin{equation} \label{eq:deltatheta}
\Delta \theta =\dfrac{2\pi}{\sqrt{1-\dfrac{\alpha^2}{L^2c^2}}}
\end{equation}
(see \cite[Formula (2.25)]{BoDaFe22}). 
Observing that $\Delta \theta >2\pi$, it is immediate to conclude that \eqref{eq:keplero} does not admits non-circular periodic solutions in $\Lambda_1$.

\noindent
For $k\geq 2$, from the previous discussion we deduce that $x$ is a $T$-periodic solution of \eqref{eq:keplero} in $\Lambda_k$ if and only if there exist $n\in \{1,\ldots, k-1\}$ and $(h,L)\in \Sigma$ such that
\begin{equation} \label{eq:periodiche}
T=nT_h,\quad n\Delta \theta =2\pi k.
\end{equation}
Notice that \eqref{eq:periodiche} implies that periodic non-circular solutions in $\Lambda_k$ are necessarily solutions of type $(n,k)$, for some $n\in \{1, \ldots, k-1\}$.
Now, recalling \eqref{eq:periodorad} and \eqref{eq:deltatheta}, it is easy to see that \eqref{eq:periodiche} is equivalent to 
\[
m^2c^4-h^4=\left(\dfrac{2\pi \alpha m^2c^3n}{T}\right)^{2/3},\quad L^2=\dfrac{\alpha^2}{c^2}\, \dfrac{k^2}{k^2-n^2}.
\]
From these relations we then obtain that $(h,L)\in \Sigma$ if and only if
\[
\dfrac{\alpha^2}{c^2}\, \dfrac{k^2}{k^2-n^2} < \dfrac{\alpha^{4/3}m^{2/3}T^{2/3}}{(2\pi n)^{2/3}},
\]
i.e.
\[
T>\dfrac{2\pi \alpha k^3 n}{mc^3 (k^2-n^2)^{3/2}}.
\]
From \eqref{eq:tstar} and \eqref{eq:ukn}, we can write this condition as $T>u^k_n$.
Therefore, \eqref{eq:keplero} has a $T$-periodic solution if and only if $T>u^k_n$, for some $n\in \{1, \ldots, k-1\}$, and this solution is of type $(n,k)$. This is sufficient to conclude.
\end{proof}
\noindent
The aim of the next sections is to study the minimality of the circular and non-circular periodic solutions of \eqref{eq:keplero} given in Propositions \ref{pro:circolari} and \ref{pro:rosette}, depending on the values of $T>0$.
In particular, we plan to detect which is the minimal solution whose existence has been proved in Section \ref{sec3.2}.
To this end, we will be able to explicitly compute and compare the action levels of the non-circular solutions (see Section \ref{subsec:rosette}). The computation of the action level of circular solutions is still possible, but it it not evident how this level is related to the levels of non-circular solutions for any $T>0$.
However, computing the Morse index of circular solutions (see Section \ref{subsec:circ}) will be sufficient to conclude (Section \ref{subsec:minimalita}). 
\subsection{Action level of non-circular solutions}\label{subsec:rosette}
In this section we prove the following result on the action level of non-circular solutions.
\begin{Proposition} \label{pro:azionerosetta}
	Let $T>0$, $k\in \NN$, with $k\geq 2,$ $i_T\in \{1,\ldots, k-1\}$ such that \eqref{eq:condTrosetta} is satisfied and let $x$ be a solution of \eqref{eq:keplero} of type $(n,k)$, for some $n\in \{1, \ldots, i_T\}$.
	Then, the action level $I(x)$ of $x$ is given by
	\begin{equation} \label{eq:livrosetta}
	I_n^k=mc^2T+\dfrac{2\pi \alpha}{c}\, \sqrt{k^2-n^2}-\dfrac{T}{c}\, \left(m^{2/3}c^2-\left(\dfrac{2\pi \alpha n}{T}\right)^{2/3}\right)^{3/2}.
	\end{equation}
\end{Proposition}
\begin{Remark} \label{rem:angolo} Let us point out some relations for the radial and angular components $(r,\theta)$ of solutions $x$ of \eqref{eq:keplero}, which have been proved in \cite{BoDaFe22}: denoting again by $h$ and $L$ the energy and angular momentum of $x$, respectively (see \eqref{eq:energia} and \eqref{eq:momento}), we then have
	\begin{equation} \label{eq:thetaprimo}
	\dot{\theta}=\dfrac{L\, \sqrt{1-\dfrac{|\dot{x}|^2}{c^2}}}{m|x|^2}
	\end{equation}
and
	\begin{equation} \label{eq:rditheta}
	r(\theta)=\dfrac{1}{B}\, \dfrac{1}{1+E\cos \left(\sqrt{1-\dfrac{\alpha^2}{L^2c^2}}\, \theta \right)},
	\end{equation}
where
	\begin{equation} \label{eq:costanti1}
	B=\dfrac{\alpha h}{L^2c^2-\alpha^2},\quad e=\dfrac{\sqrt{\alpha^2 m^2 c^4+(h^2-m^2c^4)L^2c^2}}{L^2c^2-\alpha^2},\quad E=\dfrac{e}{B}.
	\end{equation}
\end{Remark}
\begin{proof}
Let us first recall that
\[
I(x)=\int_{0}^{T} mc^2\left(1-\sqrt{1-\dfrac{|\dot{x}|^2}{c^2}}\right)+\dfrac{\alpha}{|x|}=mc^2T-mc^2 \int_0^T \sqrt{1- \dfrac{|\dot{x}|^2}{c^2}}+\int_0^T \dfrac{\alpha}{|x|}.
\]
From the conservation of the energy (see \eqref{eq:energia}) we infer that
\begin{equation} \label{eq:conti1}
\sqrt{1-\dfrac{|\dot{x}|^2}{c^2}}=\dfrac{mc^2}{h+\alpha/|x|},
\end{equation}
which implies that
\[
I(x)=mc^2T-mc^2 \int_0^T \dfrac{mc^2}{h+\alpha/|x|}+\int_0^T \dfrac{\alpha}{|x|}.
\]
Now, let us notice that from \eqref{eq:thetaprimo} we deduce that $\theta$ is monotone; hence, it is possible to change the variable of integration setting $t=t(\theta)$, where $t(\theta)$ is the inverse function of $\theta$. Taking again into account \eqref{eq:thetaprimo} and \eqref{eq:conti1}, we infer that the function $t$ satisfies
\[
\dfrac{dt}{d\theta}=\dfrac{1}{\dot{\theta}}=\dfrac{m|x|^2}{L\, \sqrt{1-\dfrac{|\dot{x}|^2}{c^2}}}=\dfrac{h|x|^2+\alpha |x|}{Lc^2}.
\]
On the other hand, recalling \eqref{eq:deltatheta}, \eqref{eq:periodiche} and assuming without loss of generality $\theta(0)=0$, we have $\theta(T)=2\pi k$.
Using \eqref{eq:conti1} and \eqref{eq:rditheta}, we then obtain
\begin{align*}
I(x) & =  mc^2T-\dfrac{m^2c^2}{L}\, \int_0^{2\pi k} |x|^2\, d\theta+\dfrac{\alpha}{Lc^2}\, \int_0^{2\pi k} (h|x|+\alpha)\, d\theta  \\
& =  mc^2T +\dfrac{2\pi k \alpha^2}{Lc^2}-\dfrac{m^2c^2}{B^2 L}\, \int_0^{2\pi k} \dfrac{1}{\left(1+E\cos \left(\sqrt{1-\dfrac{\alpha^2}{L^2c^2}}\, \theta \right)\right)^2}\, d\theta\\
& \hphantom{=} + \dfrac{\alpha h}{B Lc^2}\, \int_0^{2\pi k} \dfrac{1}{1+E\cos \left(\sqrt{1-\dfrac{\alpha^2}{L^2c^2}}\, \theta \right)}\, d\theta.
\end{align*}
Now, from \eqref{eq:deltatheta} and \eqref{eq:periodiche} we know that $\sqrt{1-\alpha^2/L^2 c^2}=n/k$; hence, via the substitution $u=n\theta/k$, we obtain
\begin{equation} \label{eq:conti2}
\begin{aligned}
I(x) &= mc^2T +\dfrac{2\pi k \alpha^2}{Lc^2}-\dfrac{m^2c^2 k}{n B^2 L}\,
\int_0^{2\pi n} \dfrac{1}{\left(1+E\cos u\right)^2}\, du \displaystyle + \dfrac{\alpha hk}{nB Lc^2}\, \int_0^{2\pi n} \dfrac{1}{1+E\cos u}\, du \\
& = mc^2T +\dfrac{2\pi k \alpha^2}{Lc^2}-\dfrac{m^2c^2 k}{B^2 L}\, \int_0^{2\pi} \dfrac{1}{\left(1+E\cos u\right)^2}\, du \displaystyle + \dfrac{\alpha hk}{B Lc^2}\, \int_0^{2\pi} \dfrac{1}{1+E\cos u}\, du.
\end{aligned}
\end{equation}
The integrals in \eqref{eq:conti2} can be computed using the residue theorem, obtaining
\[
\int_0^{2\pi} \dfrac{1}{1+E\cos u}\, du=\dfrac{2\pi}{(1-E^2)^{1/2}},\quad \int_0^{2\pi} \dfrac{1}{\left(1+E\cos u\right)^2}\, du=\dfrac{2\pi}{(1-E^2)^{3/2}}.
\]
We then deduce
\begin{equation} \label{eq:conti3}
I(x)=\displaystyle mc^2T +\dfrac{2\pi k \alpha^2}{Lc^2}-\dfrac{2\pi k m^2c^2 }{B^2 L (1-E^2)^{3/2}}+ \dfrac{2\pi k\alpha h}{B Lc^2 (1-E^2)^{1/2}}.
\end{equation}
From \eqref{eq:deltatheta}, \eqref{eq:periodiche}, \eqref{eq:costanti1} it is easy to prove that
\begin{align*}
& L=\dfrac{\alpha}{c}\, \dfrac{k}{\sqrt{k^2-n^2}},\quad B=\dfrac{h}{\alpha}\, \dfrac{k^2-n^2}{n^2},\\
& e=\dfrac{\sqrt{k^2-n^2}}{\alpha n^2}\, \sqrt{h^2 k^2-m^2c^4n^2},\quad E=\dfrac{1}{h\sqrt{k^2-n^2}} \, \sqrt{h^2 k^2-m^2c^4n^2},\\
& (1-E^2)^{1/2}=\dfrac{n}{h}\, \sqrt{\dfrac{m^2c^4-h^2}{k^2-n^2}},\quad (1-E^2)^{3/2}=\dfrac{n^3}{h^3}\, \left(\dfrac{m^2c^4-h^2}{k^2-n^2}\right)^{3/2}.
\end{align*}
Replacing in \eqref{eq:conti3} we obtain
\begin{equation} \label{eq:conti4}
I(x) = mc^2T +\dfrac{2\pi  \alpha}{c}\, \sqrt{k^2-n^2}-\dfrac{2\pi \alpha n h^3}{c\, (m^2c^4-h^2)^{3/2}}.
\end{equation}
Taking into account \eqref{eq:periodorad} and \eqref{eq:periodiche}, we deduce that
\[
\dfrac{T}{n}=\dfrac{2\pi \alpha m^2c^3}{c\, (m^2c^4-h^2)^{3/2}};
\]
this implies that
\[
h=m^{2/3} c \, \left(m^{2/3}c^2-\left(\dfrac{2\pi \alpha n}{T}\right)^{2/3}\right)^{1/2}.
\]
From \eqref{eq:conti4} and the relation we conclude that
\begin{eqnarray*}
I(x) &= &mc^2T +\dfrac{2\pi  \alpha}{c}\, \sqrt{k^2-n^2}-\dfrac{2\pi \alpha n }{c}\, m^2 c^3 \, \left(m^{2/3}c^2-\left(\dfrac{2\pi \alpha n}{T}\right)^{2/3}\right)^{3/2}\, \dfrac{T}{2\pi \alpha n m^2 c^3}\\
& = &mc^2T+\dfrac{2\pi \alpha}{c}\, \sqrt{k^2-n^2}-\dfrac{T}{c}\, \left(m^{2/3}c^2-\left(\dfrac{2\pi \alpha n}{T}\right)^{2/3}\right)^{3/2}. \qedhere
\end{eqnarray*}
\end{proof}
\noindent
In order to study the minimality of non-circular solutions, it will be useful to compare the action levels of the solutions corresponding to different integers $n$.
This is the result of the following proposition.
\begin{Proposition} \label{prop:azionecrescente}
	Let $T>0$, $k\in \NN$, with $k\geq 2,$ $i_T\in \{1,\ldots, k-1\}$ such that \eqref{eq:condTrosetta} is satisfied and let $I^k_n$ be given in \eqref{eq:livrosetta} for $n\in \{1, \ldots, i_T\}$.
	Then, we have
	\[
	I_1^k\leq I^k_2\leq \ldots \leq I^k_{i_T}.
	\]
	\end{Proposition}
\begin{proof}
	Let us first define $\varUpsilon: [0,k]\to \RR$ by
	\[
	\varUpsilon (s)=mc^2T+\dfrac{2\pi \alpha}{c}\, \sqrt{k^2-s^2}-\dfrac{T}{c}\, \left(m^{2/3}c^2-\left(\dfrac{2\pi \alpha s}{T}\right)^{2/3}\right)^{3/2},\quad \forall \ s\in [0,k].
	\]
	From Proposition~\ref{pro:azionerosetta} we know that 
	\begin{equation} \label{eq:conton0}
	\varUpsilon (n)=I_n^k,\quad \forall \ n\in \{1, \ldots, i_T\}.
	\end{equation}
	A simple computation shows that
	\[
	\varUpsilon'(s)=\dfrac{(2\pi \alpha)^{2/3}\, T^{1/3}}{c}\, \left[\left(\dfrac{m^{2/3}c^2}{s^{2/3}}-\left(\dfrac{2\pi \alpha}{T}\right)^{2/3}\right)^{1/2}-\left(\dfrac{2\pi \alpha}{T}\right)^{1/3}\, \dfrac{s}{\sqrt{k^2-s^2}}\right],
	\]
for every $s\in (0,k]$.
From this relation we deduce that
\[
\varUpsilon'(s) >0 \quad \Longleftrightarrow \quad \dfrac{m^{2/3}c^2}{s^{2/3}}-\left(\dfrac{2\pi \alpha}{T}\right)^{2/3}>\left(\dfrac{2\pi \alpha}{T}\right)^{2/3}\, \dfrac{s^2}{k^2-s^2}.
\]
A straightforward computation proves that
\begin{equation} \label{eq:conton1}
\varUpsilon'(s) >0 \quad \Longleftrightarrow \quad   m^{2/3}c^2 s^2 + \left(\dfrac{2\pi \alpha}{T}\right)^{2/3}\, k^2\, s^{2/3} -m^{2/3}c^2 k^2<0
\end{equation}
Define now 
\begin{equation} \label{eq:conton2}
\zeta (s)=m^{2/3}c^2 s^2 + \left(\dfrac{2\pi \alpha}{T}\right)^{2/3}\, k^2\, s^{2/3} -m^{2/3}c^2 k^2, \quad \forall \ s\in [0,k];
\end{equation}
the function $\zeta$ is clearly strictly increasing in $[0,k]$ and satifies 
\[
\zeta(0)=-m^{2/3}c^2 k^2<0,\quad \zeta (k)=\left(\dfrac{2\pi \alpha}{T}\right)^{2/3}\, k^2\, k^{2/3}>0.
\]
Hence, $\zeta$ has a unique zero $\zeta_0$ in $[0,k]$ and
\[
\zeta(s)<0\quad \Longleftrightarrow \quad 0<s<\zeta_0.
\]
Taking into account \eqref{eq:conton1} and \eqref{eq:conton2}, we then obtain that 
\[
\varUpsilon'(s) >0 \quad \Longleftrightarrow \quad 0<s<\zeta_0.
\]
Recalling \eqref{eq:conton0}, to prove the thesis it is now sufficient to show that $i_T<\zeta_0$, i.e. $\zeta(i_T)<0$.
Indeed, this is a consequence of assumption \eqref{eq:condTrosetta}: if $T>u^k_{i_T}$, from \eqref{eq:tstar} and \eqref{eq:ukn} we deduce that
\begin{equation} \label{eq:conton3}
T>\dfrac{i_T \, k^3}{(k^2-i_T^2)^{3/2}}\, \dfrac{2\pi \alpha}{mc^3}.
\end{equation}
On the other hand, from \eqref{eq:conton2} we have
\[
\zeta (i_T)=m^{2/3}c^2 i_T^2 + \left(\dfrac{2\pi \alpha}{T}\right)^{2/3}\, k^2\, i_T^{2/3} -m^{2/3}c^2 k^2.
\]
Now, \eqref{eq:conton3} implies
\[
\zeta (i_T) < m^{2/3}c^2 (i_T^2-k^2)+(2\pi \alpha)^{2/3}\, \left(\dfrac{mc^3}{2\pi \alpha}\, \dfrac{(k^2-i_T^2)^{3/2}}{i_T \, k^3}\right)^{2/3}\, k^2\, i_T^{2/3}=0. \qedhere
\]
\end{proof}
\subsection{Morse index of circular solutions}\label{subsec:circ}
The minimality of the circular solution $x_C$ given in \eqref{eq:solcirc} can be studied by computing its Morse index $\iota (x_C)$,
which can be defined by using the lagrangian structure of \eqref{eq:keplero}.
Indeed, as already noticed problem \eqref{eq:keplero} is the Euler-Lagrange equation corresponding to the Lagrangian
\[
\mathcal{L}(x,\dot{x})=mc^2 \left(1-\sqrt{1-\dfrac{|\dot{x}|^2}{c^2}}\right)+ \dfrac{\alpha}{|x|}.
\]
Given a circular solution $x_C$ of \eqref{eq:keplero}, linearizing the equation at $x_C$ and taking into account that
\[
\frac{\partial^2 \mathcal{L}}{\partial \dot{x}\, \partial x}=0,
\]	
we obtain
\begin{equation} \label{eq:morse11}
\frac{d}{dt}\left(A(t)\dot{q}\right)=B(t)q,
\end{equation} 
where
\[
A(t)=\frac{m}{c^2}\, \left(1-\dfrac{|\dot{x}_C(t)|^2}{c^2}\right)^{-3/2} \, x_C(t)\otimes x_C(t)+m \, \left(1-\dfrac{|\dot{x}_C(t)|^2}{c^2}\right)^{-1/2} \, \mbox{Id}
\]
and
\[
B(t)=3\alpha |x_C(t)|^{-5} \, x_C(t)\otimes x_C(t)-\alpha |x_C(t)|^{-3}\,  \mbox{Id},
\]
for every $t\in [0,T]$.
Here $x\otimes y = x y^{\top}$ when $x$ and $y$ are column vector of the same dimension.
We now observe that \eqref{eq:morse11} is the Euler-Lagrange equation of the quadratic Lagrangian
\[
L_C(t,q,\dot{q})=\dfrac{1}{2} \langle A(t)\dot{q},\dot{q}\rangle +\dfrac{1}{2}\langle B(t)q,q\rangle,
\] 
which induces the quadratic form $\tau_C: H^1_T(0,T)\to \RR$ defined by
\begin{equation} \label{eq:morsequadr}
\tau_C(q)=\int_0^T L_C(t,q(t),\dot{q}(t)),\quad \forall  \ q\in H^1_T(0,T).
\end{equation}
Noticing that $A(t)$ is positive definite, for every $t\in [0,T]$, the Morse index $j_C$ of $\tau_C$ is well-defined (cf. \cite[Par. 3.4]{Abb01}) and the Morse index $\iota (x_C)$ of the circular solution $x_C$ is defined as $j_C$. 

In the following theorem we give the result on the computation of $\iota (x_C)$.
\begin{Theorem} \label{teo:morse}
	Let $T>0$, $k\in \NN,$ with  $k\geq 1$, and $i_T\in \{0,\ldots, k-1\}$ such that \eqref{eq:condTrosetta} is fulfilled.
    Then, we have
	\[
	\iota(x_C)=2i_T.
	\]
\end{Theorem}
\begin{proof}
The proof is based on the arguments in \cite{BoDaFe22} and \cite{KaOfPo21}.
Following \cite{BoDaFe22}, we first define
\begin{equation*}
p=\dfrac{m\dot{x}}{\sqrt{1-\dfrac{|\dot{x}|^2}{c^2}}}
\end{equation*}
and 
\begin{equation*}
\Omega=(0,+\infty)\times \mathbb{T}^1 \times \mathbb{R}^2;
\end{equation*}
Via the diffeomorphism
\[
\Psi \colon \Omega\to (\mathbb{R}^2\setminus\{0\})\times \mathbb{R}^2, \qquad (r,\vartheta,l,\Phi)\mapsto (x,p), 
\]
given by 
\[
x=re^{i\vartheta}, \qquad p=le^{i\vartheta}+\dfrac{\Phi}{r}ie^{i\vartheta},
\]
\eqref{eq:keplero} can be written in the form 
\begin{equation}\label{eq:hamkeplero}
\begin{cases}
\dot r = \partial_l \mathcal{H}_0(r,\vartheta,l,\Phi) = \dfrac{l}{m}\dfrac{1}{\sqrt{1+\dfrac{l^2+\Phi^2/r^2}{m^2c^2}}}, \vspace{7pt}\\
\dot \vartheta = \partial_\Phi \mathcal{H}_0(r,\vartheta,l,\Phi) = \dfrac{\Phi}{mr^{2}}\dfrac{1}{\sqrt{1+\dfrac{l^2+\Phi^2/r^2}{m^2c^2}}}, \vspace{7pt}\\
\dot l = - \partial_r \mathcal{H}_0(r,\vartheta,l,\Phi) = \dfrac{\Phi^{2}}{m r^{3}} \dfrac{1}{\sqrt{1+\dfrac{l^2+\Phi^2/r^2}{m^2c^2}}} - \dfrac{\alpha}{r^{2}}, \vspace{7pt}\\
\dot \Phi = -\partial_\vartheta \mathcal{H}_0(r,\vartheta,l,\Phi) = 0, \vspace{2pt}
\end{cases}
\end{equation}
corresponding to the Hamiltonian
\begin{equation*}
\mathcal{H}_0(r,\vartheta,l,\Phi)=mc^2\sqrt{1+\dfrac{l^2+\Phi^2/r^2}{m^2c^2}}-\dfrac{\alpha}{r}.
\end{equation*}
It is immediate to see that the circular solution $x_C$ given in \eqref{eq:solcirc} corresponds to $\varsigma(t)=(R,\omega t, 0, L)$, $t\in [0,T]$, where $L$ is as in \eqref{eq:valoreL}.

According to the Morse index theorem \cite[Cor. 3.4.2]{Abb01}, we have
\begin{equation} \label{eq:morse}
\iota(x_C)=i^{\textrm{CZ}} (\zeta; [0,T]),
\end{equation}
where $\zeta$ is the fundamental matrix solution of the linearized system of \eqref{eq:hamkeplero} at $\varsigma$ and $i^{\textrm{CZ}} (\zeta; [0,T])$ is its Conley-Zehnder index on the period $[0,T]$.
A standard computation shows that the linearized Hamiltonian system of \eqref{eq:hamkeplero} at $\varsigma$ is the linear system $\dot w=Mw$, where 
\[
M=\left[\begin{array}{cccc}
0 & 0 & A & 0 \vspace{4pt} \\
B & 0 & 0 & C \vspace{4pt} \\
D & 0 & 0 & -B \vspace{4pt} \\
0 & 0 & 0 & 0 
\end{array}
\right]
\]
where
\[
A=\dfrac{\sqrt{L^2c^2-\alpha^2}}{mLc},\quad 
B=\dfrac{\alpha^3m^2(\alpha^2-2L^2c^2)}{L^5(L^2c^2-\alpha^2)},\quad
C=\dfrac{\alpha^2 m\sqrt{L^2c^2-\alpha^2}}{L^5c},\quad 
D=-\dfrac{\alpha^4m^3c}{L^5\sqrt{L^2c^2-\alpha^2}}.
\]
The matrix $M$ has the eigenvalues $\{0, 0, \pm i\omega'\}$, where
\begin{equation}\label{eq:autovalori}
\omega'=\dfrac{\alpha^2m}{L^3}.
\end{equation}
According to \cite[Sect. 4.1]{KaOfPo21}, there exists a symplectic matrix $P\in \textrm{Sp}(4)$ such that 
\[
M=P\, \left(\left[\begin{array}{cc}
0 & s_L \vspace{4pt} \\
0 & 0
\end{array}
\right]
\diamond \left[\begin{array}{cc}
0 & -\omega' \vspace{4pt}\\
\omega'& 0
\end{array}
\right]
\right)\, P^{-1},
\]
for some $s_L>0$.
As a consequence, the Conley-Zehnder index of $\zeta$ on the period $[0,T]$ is given by
\begin{equation} \label{eq:indice1}
i^{\textrm{CZ}} (\zeta; [0,T])=i^{\textrm{CZ}} (\eta; [0,T])+i^{\textrm{CZ}} (\xi; [0,T]),
\end{equation}
where
\[
\eta(t)=\left[\begin{array}{cc}
1 & s_L t \vspace{4pt}\\
0 & 1
\end{array}
\right],\quad \xi(t)=\exp\left( t\left[\begin{array}{cc}
0 & -\omega' \vspace{4pt}\\
\omega'& 0
\end{array}
\right]\right),
\]
for every $t\in [0,T]$. Taking into account the fact that $s_L>0$, from \cite[Lemma 3.7]{KaOfPo21} we deduce that
\begin{equation} \label{eq:indice2}
i^{\textrm{CZ}} (\eta; [0,T])=-1.
\end{equation}

Now, from Proposition 2.10 and Lemma 3.5 in \cite{KaOfPo21} we infer that 
\begin{equation} \label{eq:indice3}
i^{\textrm{CZ}} (\xi; [0,T])=
\begin{cases}
2\dfrac{T\omega'}{2\pi}-1 &\text{if } \dfrac{T\omega'}{2\pi}\in \NN \vspace{4pt} \\
2\Bigl\lfloor \dfrac{T\omega'}{2\pi}\Bigr\rfloor +1 & \text{if } \dfrac{T\omega'}{2\pi}\notin \NN. 
\end{cases}
\end{equation}
Now, recalling \eqref{eq:valoreL} and \eqref{eq:autovalori}, we have
\[
\frac{T\omega'}{2\pi}=\frac{T\alpha^2 m}{2\pi L^3}=
\frac{k\sqrt{L^2c^2-\alpha^2}}{Lc} = k\sqrt{1-\frac{\alpha^2}{L^2c^2}}<k.
\]
As a consequence, for $n\in \{0, \ldots, k-1\}$
by means of an easy computation we deduce that
\[
n\leq \dfrac{T\omega'}{2\pi} < n+1 \quad \Longleftrightarrow \quad \dfrac{\alpha^2}{c^2}\, \dfrac{k^2}{k^2-n^2}\leq L^2 < \dfrac{\alpha^2}{c^2}\, \dfrac{k^2}{k^2-(n+1)^2}.
\]
Using \eqref{eq:valoreL} and recalling \eqref{eq:ukn}, we conclude that
\[
n\leq \dfrac{T\omega'}{2\pi} < n+1 \quad \Longleftrightarrow \quad u^k_n\leq T<u^k_{n+1}.
\]
As a consequence, from \eqref{eq:indice3} we deduce that
\begin{equation} \label{eq:indice4}
i^{\textrm{CZ}} (\xi; [0,T])=
\begin{cases}
2n-1 & \text{if } T=u^k_n \\
2n+1 & \text{if } T\in (u^n_k,u^{n+1}_k). 
\end{cases}
\end{equation}
From \eqref{eq:morse}, \eqref{eq:indice1}, \eqref{eq:indice2} and \eqref{eq:indice4} we obtain the result.
\end{proof}
\begin{Remark} \label{rem:morsenozero}
Let us notice that if $\iota(x_C)\geq 1$ then $x_C$ is not minimal.
Indeed, a simple computation shows that, as in the classical case, the Morse index $\iota(x_C)$ is equal to the sum of the dimensions of the eigenspaces relative to the negative eigenvalues of the periodic  problem
\begin{equation} \label{eq:morseeig}
\begin{cases}
\dfrac{d}{dt}\left(A(t)\dot{q}\right)=B(t)q -\lambda q \vspace{4pt} \\
q(0)=q(T),\ \dot{q}(0)=\dot{q}(T).
\end{cases}
\end{equation}
Moreover, $\tau_{C}(q)<0$ if $q$ is a nontrivial function belonging to the direct sum of those
eigenspaces, where $\tau_C$ is defined in \eqref{eq:morsequadr}.
Therefore, if $\iota(x_C)\geq 1$, then there exists an eigenfunction $\bar q\in H^1_T(0,T)$ associated to \eqref{eq:morseeig} and such that
\begin{equation} \label{eq:tauneg}
\tau_C(\bar q)<0.
\end{equation}
A regularity argument shows that $\bar q\in C^1(0,T)$, thus implying that $x_C+\epsilon \bar q\in D_{\psi}\cap \Lambda_k$ if $\epsilon$ is sufficiently small (cf. Section \ref{sec3.1} for the notation).
Now, from Taylor's formula we deduce that
\[
I(x_C+\epsilon \bar q)-I(x_C)=\dfrac{1}{2}\ \tau_C(\bar q)\, \epsilon^2 +o(\epsilon^2),\quad \epsilon \to 0.
\]
Taking into account \eqref{eq:tauneg}, we conclude that $x_C$ is not minimal.
\end{Remark}	
\subsection{Minimality and asymptotic level estimates}\label{subsec:minimalita}
We are now in a position to prove the result on the minimality of solutions of \eqref{eq:keplero}.
\begin{Theorem} \label{teo:minimalita}
	Let $T>0$ and $k\in \NN,$ with $k\geq 1$. Then, we have:
	\begin{enumerate}
		\item if $k=1$, for every $T>0$ the minimal solution is the circular solution defined in \eqref{eq:solcirc};
		\item if $k\geq 2$ and $T\leq u^k_1$, the minimal solution is the circular solution defined in \eqref{eq:solcirc};
		\item if $k\geq 2$ and $T> u^k_1$, the minimal solution is the non-circular solution of type $(1,k)$ given in Proposition~\ref{pro:rosette}. 
	\end{enumerate}
\end{Theorem}
\begin{proof}
We have already observed that for $k=1$ the unique $T$-periodic solution in $\Lambda_k$ is the circular solution.
This is the situation also when $k\geq 2$ and $T\leq u^k_1$, which corresponds to the case $i_T=0$ in Proposition~\ref{pro:rosette}.
	
In the case $k\geq 2$ and $i_T>0$, from Theorem \ref{teo:morse} and Remark \ref{rem:morsenozero} we deduce that the circular solution cannot be minimal.
Hence, the minimal solution is one of the non-circular solutions; Proposition \ref{prop:azionecrescente} allows to conclude that the minimal solution is the solution of type $(1,k)$.
\end{proof}
\begin{Remark} Let us observe, as a complement of the study of the values of the action functional, that when $k\geq 2$ the action level of the circular solution is greater of all the levels of the existing non-circular solutions, at least for $T$ sufficiently large.

Indeed, let us first fix $k\geq 2$ and observe that if $T>u^k_{k-1}$ then \eqref{eq:keplero} has $k-1$ solutions of type $(1,k), \ldots, (k-1,k)$ (cf. Proposition~\ref{pro:rosette}). 
As a consequence, we can consider the action level of the solution $x$ of type $(n,k)$, for every $n\in \{1,\ldots, k-1\}$, which is given in \eqref{eq:livrosetta}. Now, let us notice that
\begin{align*}
\left(m^{2/3}c^2-\left(\frac{2\pi \alpha n}{T}\right)^{2/3}\right)^{3/2} & =mc^3\, \left(1-\left(\frac{2\pi \alpha n}{mc^3}\right)^{2/3}\, \frac{1}{T^{2/3}}\right)^{3/2}\\
&= mc^3\, \left(1-\frac{3}{2}\, \left(\frac{2\pi \alpha n}{mc^3}\right)^{2/3}\, \frac{1}{T^{2/3}}+\frac{15}{8}\, \left(\frac{2\pi \alpha n}{mc^3}\right)^{4/3}\, \frac{1}{T^{4/3}} +O\left(\frac{1}{T^{2}}\right)\right),
\end{align*}
as $ T\to +\infty.$
Substituting this relation in \eqref{eq:livrosetta} we obtain
\begin{equation} \label{eq:livrosetta2}
\begin{aligned}
I_n^k&=mc^2T+\frac{2\pi \alpha}{c}\, \sqrt{k^2-n^2}\\
&\hphantom{=\,} -mc^2 T\, \left(1-\dfrac{3}{2}\, \left(\frac{2\pi \alpha n}{mc^3}\right)^{2/3}\, \frac{1}{T^{2/3}}+\frac{15}{8}\, \left(\frac{2\pi \alpha n}{mc^3}\right)^{4/3}\, \frac{1}{T^{4/3}} +O\left(\frac{1}{T^{2}}\right)\right)\\
& = \frac{3}{2}\, (2\pi \alpha n )^{2/3}m^{1/3}\, T^{1/3} +\frac{2\pi \alpha}{c}\, \sqrt{k^2-n^2}
-\frac{15}{8}\, \frac{(2\pi \alpha n)^{4/3}}{m^{1/3}c^2}\, \frac{1}{T^{1/3}} +O\left(\frac{1}{T}\right),\quad T\to +\infty.
\end{aligned}
\end{equation}
On the other hand, recalling \eqref{eq:solcirc} and \eqref{eq:valoreL}, the action level of the circular solution $x_C$ given in Proposition~\ref{pro:circolari} is
\begin{equation} \label{eq:livcirc}
\begin{aligned}
I(x_C) & = mc^2T-\int_0^T mc^2\,  \sqrt{1-\frac{|\dot{x}|^2}{c^2}}-\frac{\alpha}{|x|}
= mc^2T-\int_0^T mc^2 \, \sqrt{1-\frac{\alpha^2}{L^2 c^2}}-\frac{\alpha}{R}\\
& = mc^2 T - \frac{mcT}{L}\, \sqrt{L^2c^2-\alpha^2}+\frac{mc\alpha^2 T}{L\sqrt{L^2c^2-\alpha^2}}
=mc^2 T - \frac{m^2 c^2\alpha^2 }{2\pi kL^3}\, T^2+2\pi k L.
\end{aligned}
\end{equation}
Now, from \eqref{eq:valoreL} we deduce that
\[
L^4(L^2c^2-\alpha^2)=\dfrac{m^2c^2\alpha^4}{4\pi^2 k^2}\, T;
\]
hence, $x=L^2$ is a solution of the equation
\begin{equation} \label{eq:terzogrado}
4\pi^2 k^2 c^2x^3-4\pi^2k^2\alpha^2 x^2-m^2c^2\alpha^4 T^2=0.
\end{equation}
The substitution
\begin{equation} \label{eq:sostxy}
y=12\pi^2 k^2 c^2 x-4\pi^2k^2 \alpha^2
\end{equation}
transforms \eqref{eq:terzogrado} in the form
\[
y^3+py+q=0,\quad p=-48\pi^4 k^4 \alpha^4,\ q=-16\pi^4 k^4 \alpha^4(27m^2 c^6T^2+8\pi^2 k^2 \alpha^2).
\]
An application of Cardano's formula gives
\[
y=6\pi^{4/3}k^{4/3}\alpha^{4/3} m^{2/3}c^2 T^{2/3}U(T),
\]
where
\begin{equation} \label{eq:UT}
U(T)=\sqrt[3]{1+\frac{8\pi^2 k^2 \alpha^2}{27 m^2 c^6 T^2}-\sqrt{1+\frac{16\pi^2 k^2 \alpha^2}{27 m^2 c^6 T^2}}}+\sqrt[3]{1+\frac{8\pi^2 k^2 \alpha^2}{27 m^2 c^6 T^2}+\sqrt{1+\frac{16\pi^2 k^2 \alpha^2}{27 m^2 c^6 T^2}}}.
\end{equation}
As a consequence, recalling \eqref{eq:sostxy}, we obtain
\[
x=L^2=\frac{\alpha^2}{3c^2}\, \left(\frac{3m^{2/3}c^2 T^{2/3}}{2\pi^{2/3}k^{2/3}\alpha^{2/3}}\, U(T)+1\right)
\]
and
\[
L=\frac{\alpha}{\sqrt{3}c}\, \left(\frac{3m^{2/3}c^2 T^{2/3}}{2\pi^{2/3}k^{2/3}\alpha^{2/3}}\, U(T)+1\right)^{1/2}.
\]
From \eqref{eq:livcirc} we then deduce that
\begin{equation} \label{eq:livcirc2}
\begin{aligned}
I(x_C) & = mc^2 T - \frac{m^2 c^2\alpha^2 }{2\pi k}\, \left(\frac{\sqrt{3}c}{\alpha}\right)^3 \, T^2\, \left(\frac{3m^{2/3}c^2 T^{2/3}}{2\pi^{2/3}k^{2/3}\alpha^{2/3}}\, U(T)+1\right)^{-3/2}\\
& \hphantom{=\,} + 2\pi k \, \frac{\alpha}{\sqrt{3}c}\, \left(\frac{3m^{2/3}c^2 T^{2/3}}{2\pi^{2/3}k^{2/3}\alpha^{2/3}}\, U(T)+1\right)^{1/2}\\
& =mc^2 T -\frac{3\sqrt{3m^2 c^5}}{2\pi k\alpha} \, T^2  \left(\frac{3m^{2/3}c^2 T^{2/3}}{2\pi^{2/3}k^{2/3}\alpha^{2/3}} U(T)+1\right)^{-3/2}+\frac{2\pi k\alpha}{\sqrt{3}c} \left(\frac{3m^{2/3}c^2 T^{2/3}}{2\pi^{2/3}k^{2/3}\alpha^{2/3}} U(T)+1\right)^{1/2} \!\!\!\!.
\end{aligned}
\end{equation}
Let us now recall that the circular solution $x_C$ exists for every $T>0$.
Hence, it is possible to study the asymptotic behaviour of $I(x_C)$ when $T\to +\infty$; to this end, let us first observe that
\[
\sqrt{1+\frac{16\pi^2 k^2 \alpha^2}{27 m^2 c^6 T^2}}=1+\frac{8\pi^2 k^2 \alpha^2}{27 m^2 c^6 T^2}-\frac{32\pi^4 k^4 \alpha^4}{27^2 m^4 c^12 T^4}+O\left(\frac{1}{T^6}\right),\quad T\to +\infty.
\]
As a consequence, we have
\begin{equation} \label{eq:ven11}
\begin{aligned}
\sqrt[3]{1+\frac{8\pi^2 k^2 \alpha^2}{27 m^2 c^6 T^2}-\sqrt{1+\frac{16\pi^2 k^2 \alpha^2}{27 m^2 c^6 T^2}}} &=\sqrt[3]{\frac{32\pi^4 k^4 \alpha^4}{27^2 m^4 c^12 T^4}+O\left(\frac{1}{T^6}\right)}\\
&=\frac{32^{1/3}\pi^{4/3} k^{4/3} \alpha^{4/3}}{9 m^{4/3} c^4}\, \frac{1}{T^{4/3}}+O\left(\frac{1}{T^{10/3}}\right),\quad T\to +\infty,
\end{aligned}
\end{equation}
and
\begin{equation} \label{eq:ven22}
\begin{aligned}
\sqrt[3]{1+\frac{8\pi^2 k^2 \alpha^2}{27 m^2 c^6 T^2}+\sqrt{1+\frac{16\pi^2 k^2 \alpha^2}{27 m^2 c^6 T^2}}}
& =\sqrt[3]{2+\frac{16\pi^2 k^2 \alpha^2}{27 m^2 c^6 T^2}-\frac{32\pi^4 k^4 \alpha^4}{27^2 m^4 c^12 T^4}+O\left(\frac{1}{T^6}\right)}\\
& =\sqrt[3]{2}+O\left(\frac{1}{T^{2}}\right),\quad T\to +\infty.
\end{aligned}
\end{equation}
From \eqref{eq:UT}, \eqref{eq:ven11} and \eqref{eq:ven22} we deduce that
\[
U(T)=\sqrt[3]{2}+\frac{32^{1/3}\pi^{4/3} k^{4/3} \alpha^{4/3}}{9 m^{4/3} c^4}\, \frac{1}{T^{4/3}}+O\left(\frac{1}{T^{2}}\right),\quad T\to +\infty.
\]
We then have
\begin{equation} \label{eq:ven44}
\begin{aligned}
1 +& \frac{3m^{2/3}c^2 T^{2/3}}{2\pi^{2/3}k^{2/3}\alpha^{2/3}} U(T) \\
& = 1 + \frac{3m^{2/3}c^2 T^{2/3}}{2\pi^{2/3}k^{2/3}\alpha^{2/3}}\, \left(\sqrt[3]{2}+\frac{32^{1/3}\pi^{4/3} k^{4/3} \alpha^{4/3}}{9 m^{4/3} c^4}\, \frac{1}{T^{4/3}}+O\left(\frac{1}{T^{2}}\right)\right)\\
& = \frac{3m^{2/3}c^2}{2^{2/3}\pi^{2/3}k^{2/3}\alpha^{2/3}}\,T^{2/3}+1+\frac{4^{1/3}\pi^{2/3} k^{2/3} \alpha^{2/3}}{3 m^{2/3} c^2}\, \frac{1}{T^{2/3}}+O\left(\frac{1}{T^{4/3}}\right)\\
& = \frac{3m^{2/3}c^2}{2^{2/3}\pi^{2/3}k^{2/3}\alpha^{2/3}}\, T^{2/3}\left(1+\frac{2^{2/3}\pi^{2/3}k^{2/3}\alpha^{2/3}}{3m^{2/3}c^2}\, \frac{1}{T^{2/3}}+\frac{2^{4/3}\pi^{4/3} k^{4/3} \alpha^{4/3}}{9 m^{4/3} c^4}\, \frac{1}{T^{4/3}}+O\left(\frac{1}{T^{2}}\right)\right),
\end{aligned}
\end{equation}
for $T\to +\infty$.
This implies that 
\begin{equation} \label{eq:ven55}
\begin{aligned}
& \left(\frac{3m^{2/3}c^2 T^{2/3}}{2\pi^{2/3}k^{2/3}\alpha^{2/3}}\, U(T)+1\right)^{1/2}\\
&=\frac{3^{1/2}m^{1/3}c}{2^{1/3}\pi^{1/3}k^{1/3}\alpha^{1/3}}\, T^{1/3} \,  \left[1+\frac{1}{2}\, \left(\frac{2^{2/3}\pi^{2/3}k^{2/3}\alpha^{2/3}}{3m^{2/3}c^2}\, \frac{1}{T^{2/3}}
+\frac{2^{4/3}\pi^{4/3} k^{4/3} \alpha^{4/3}}{9 m^{4/3} c^4}\, \frac{1}{T^{4/3}}\right) \right. \\
&\hphantom{=\,} \left.-\dfrac{1}{8}\, \frac{2^{4/3}\pi^{4/3}k^{4/3}\alpha^{4/3}}{3^2 m^{4/3}c^4}\, \frac{1}{T^{4/3}}+O\left(\frac{1}{T^{2}}\right)\right]\\
&=\frac{3^{1/2}m^{1/3}c}{2^{1/3}\pi^{1/3}k^{1/3}\alpha^{1/3}}\, T^{1/3} \, \left(1+\frac{2^{2/3}\pi^{2/3}k^{2/3}\alpha^{2/3}}{6m^{2/3}c^2}\, \frac{1}{T^{2/3}}+\dfrac{2^{4/3}\pi^{4/3} k^{4/3} \alpha^{4/3}}{24 m^{4/3} c^4}\, \frac{1}{T^{4/3}}+O\left(\dfrac{1}{T^{2}}\right)\right),
\end{aligned}
\end{equation}
for $T\to +\infty$.
Similarly, from \eqref{eq:ven44} we obtain
\begin{equation} \label{eq:ven66}
\begin{aligned}
& \left(\frac{3m^{2/3}c^2 T^{2/3}}{2\pi^{2/3}k^{2/3}\alpha^{2/3}}\, U(T)+1\right)^{-3/2}\\
&\; =\frac{2\pi k\alpha}{3^{3/2} m c^3}\, \frac{1}{T} \,  \left(1-\frac{2^{2/3}\pi^{2/3}k^{2/3}\alpha^{2/3}}{2m^{2/3}c^2}\, \frac{1}{T^{2/3}}
+\frac{2^{4/3}\pi^{4/3} k^{4/3} \alpha^{4/3}}{24 m^{4/3} c^4}\, \frac{1}{T^{4/3}}+O\left(\frac{1}{T^{2}}\right)\right),\quad T\to +\infty.
\end{aligned}
\end{equation}
From \eqref{eq:livcirc2}, \eqref{eq:ven55} and \eqref{eq:ven66} we easily infer that
\begin{equation} \label{eq:livcirc3}
\begin{aligned}
& I(x_C) \\
&= mc^2 T  -\frac{3\sqrt{3m^2 c^5}}{2\pi k\alpha} \, T^2 \, \frac{2\pi k\alpha}{3^{3/2} m c^3}\, \frac{1}{T} \,  \left(1-\frac{2^{2/3}\pi^{2/3}k^{2/3}\alpha^{2/3}}{2m^{2/3}c^2}\, \frac{1}{T^{2/3}}
+\frac{2^{4/3}\pi^{4/3} k^{4/3} \alpha^{4/3}}{24 m^{4/3} c^4}\, \frac{1}{T^{4/3}}+O\left(\frac{1}{T^{2}}\right)\right)\\
&\hphantom{=\,}  +\frac{2\pi k\alpha}{\sqrt{3}c}\, \frac{3^{1/2}m^{1/3}c}{2^{1/3}\pi^{1/3}k^{1/3}\alpha^{1/3}}\, T^{1/3} 
\left(1+\frac{2^{2/3}\pi^{2/3}k^{2/3}\alpha^{2/3}}{6m^{2/3}c^2}\, \frac{1}{T^{2/3}}+\frac{2^{4/3}\pi^{4/3} k^{4/3} \alpha^{4/3}}{24 m^{4/3} c^4}\, \frac{1}{T^{4/3}}+O\left(\frac{1}{T^{2}}\right)\right) \\
&=\frac{3}{2}\, (2\pi \alpha k )^{2/3}m^{1/3}\, T^{1/3} +\frac{1}{8}\, \frac{(2\pi \alpha k)^{4/3}}{m^{1/3}c^2}\, \frac{1}{T^{1/3}} +O\left(\frac{1}{T}\right),\quad T\to +\infty.
\end{aligned}
\end{equation}
Comparing \eqref{eq:livrosetta2} and \eqref{eq:livcirc3}, recalling that non-circular solutions $x$ exist for $n=1,\ldots, k-1$, we obtain that the relation
\[
I(x_C)>I^k_{k-1}>\ldots >I^k_1
\]
holds true asymptotically for $T\to +\infty$.
\end{Remark}
\section{Appendix: proof of the min-max principle}\label{sec:appendix}
We begin by stating a deformation-tipe result which is easily adapted from the corresponding
\cite[Theorem 2.2]{LiMa04}.
\begin{Theorem}\label{thm:deformation}
Assume that $I:X\to \RR\cup\{+\infty\}$ satisfies Assumption~\ref{ass:I}, and that
$B,C \subset X$ are nonempty closed sets such that $C$ is compact, $ C \subset D_{I} $ and
$B\cap C = \emptyset$.
If there is $\epsilon>0$ such that
\begin{equation}\label{eq:nocritico}
\forall \ x\in C \quad \exists \ \xi_{x}\in X:\quad 
\Phi^{0}(x;\xi_{x}-x)+\psi(\xi_{x})-\psi(x) < -\epsilon \|\xi_{x}-x\|,
\end{equation}
then, for every $k>1$ there exist $t_{0}\in\left(0,1\right]$, $\alpha\in C^{0}([0,1]\times X,X)$
and $\varphi\in C^{0}(X,\left[0,+\infty\right))$ which satisfy the following
\begin{enumerate}
\item
$\alpha(t,D_{\psi})\subseteq D_{\psi},\quad \forall \ t\in\left[0,t_{0}\right)$;
$\quad \alpha(t,x)=x,\quad \forall \ (t,x)\in\left[0,t_{0}\right)\times B $;
\item
$\|\alpha(t,x)-x\|\le kt,\quad \forall \ (t,x)\in\left[0,t_{0}\right)\times X$;
\item
$ I(\alpha(t,x))-I(x)\le-\epsilon\varphi(x)t,\quad \forall \ (t,x)\in\left[0,t_{0}\right)\times D_{I}$;
\item
$\varphi(x)=1,\quad \forall \ x\in C$.
\end{enumerate}
\end{Theorem}
\begin{proof}
First, we observe that \eqref{eq:nocritico} implies that, actually, $\xi_{x}\in D_{\psi}$ and 
$\xi_{x}\ne x$ for all $x\in C$.
Next, the proof follows the argument of \cite[Theorem 2.2]{LiMa04} which, in fact, develops in a
neighborhood of the compact set $C$.
It is enough to observe that, for each fixed $\xi_{x}$, the function $\varUpsilon : D_{\Phi}\times X\to \RR$ defined by
\[
\varUpsilon(z,w)=\Phi^{0}(z;\xi_{x}-w)+\psi(\xi_{x})-\psi(w)+\varepsilon\|\xi_{x}-w\|,\quad \forall \ (z,w)in D_{\Phi}\times X,
\]
is upper semicontinuous on $ D_{\Phi}\times X $ and,  by \eqref{eq:nocritico}, negative for
$(z,w)=(x,x)\in C\times C$.

\noindent
Therefore, for each $x\in C$ it is possible to find a positive $\delta_{x}<\|\xi_{x}-x\|$ such that
$\Phi$ is Lipschitz continuous in $ B(x,\delta_{x}) $ (see Assumption~\ref{ass:I}.3) and still
\[
\Phi^{0}(z;\xi_{x}-w)+\psi(\xi_{x})-\psi(w)+\varepsilon\|\xi_{x}-w\|<0,
\quad \forall \ (z,w)\in B(x,\delta_{x}).
\]
From here one can continue as in the proof of \cite[Theorem 2.2]{LiMa04}.
\end{proof}

\noindent
We now observe that we do not need the full extent of Assumption~\ref{ass:I} to prove Theorem~\ref{thm:deformation}.
In fact, Assumption~\ref{ass:I}.4 and the continuity of $\psi$ on the compact sets in which $\psi$ is
bounded, in Assumption~\ref{ass:I}.2, are not used in the proof above.

\noindent
Moreover, we remark that the third statement of the deformation theorem grants that
$ \alpha(t,D_{I}) \subseteq D_{I} $ for all $ t\in\left[0,t_{0}\right)$.
\bigskip

\noindent
Let us come now to the proof of Theorem~\ref{thm:minmax} and fix a few pieces of notation, in order to
avoid the singularity of the functional. 
Thanks to Assumption~\ref{ass:I}.4, there is $\delta_{0}>0$ such that $I(x)>c+1$ whenever
$x\in D_{I}$ and $\dist(x,\partial D_{\Phi})<2\delta_{0}$.
We set $D_{0} = \{x\in D_{\Phi}:\dist(x,\partial D_{\Phi})\ge 2\delta_{0}\}$: hence,
$x\in D_{0}$ whenever $I(x)\le c+1$.

\noindent
From now on, the argument is almost identical to that of \cite[Theorem~3.1]{LiMa04} and we only need
to check that it can be adjusted with the new assumptions on $\mathcal{F}$ and the singularity of $I$
(see Assumption~\ref{ass:I}.4).
It is enough to show that, if we fix any $\varepsilon$ and any $A_{\varepsilon}\in\mathcal{F}$
such that
\begin{equation}\label{eq:minseqA}
0<\varepsilon\le\min\{1,\delta_{0}\} \qquad \text{and} \qquad c\le \sup_{A_{\varepsilon}} I < c+\frac{\varepsilon^{2}}{8} < c+1,
\end{equation}
then there exists $x_{\varepsilon}\in X$ such that
\begin{align*}
& \Phi^{0}(x_{\varepsilon};z-x_{\varepsilon})+\psi(z)-\psi(x_{\varepsilon}) \ge -5\varepsilon\|z-x_{\varepsilon}\|, \quad \forall \ z\in X;                      \\
& \dist(x_{\varepsilon},A_{\varepsilon}) \le \frac{\varepsilon}{2};              \\
& c-\frac{\varepsilon^{2}}{8}\le I(x_{\varepsilon})\le c+\frac{5}{2}\varepsilon^{2};\\
& \dist(x_{\varepsilon},F)\le\frac{3}{2}\varepsilon \quad \text{if, moreover, } \inf_{F}I=c.
\end{align*}
In particular, \eqref{eq:minseqA} guarantess that $A_{\varepsilon}\subset D_{I}\cap D_{0}$.

\noindent
Assumptions 2 and 3 of the theorem imply that $\inf_{F}I\le c$ and, thus, we split the argument into two cases.
\subsection*{The first case: $\mathbf{\operatornamewithlimits{\mathbf{inf}}_{F} I = c}$.}
We denote $N_{\varepsilon}(F)=\{x\in X:\dist(x,F)<\varepsilon\}$ and
$G_{\varepsilon}=(\{0\}\times X)\cup([0,1]\times(( A_{\varepsilon}\setminus N_{\epsilon}(F))\cup B))$.

\noindent
Let $\overline{\eta}(t,x)=x$ for all $(t,x)\in[0,1]\times X$ and
\[
\mathcal{L}_{\varepsilon} \vcentcolon= \left\{ \eta\in C^{0}([0,1]\times X,X) :
\eta(t,x)=x \; \forall (t,x)\in G_{\varepsilon} \text{ and }
\rho(\eta,\overline{\eta}) \le \delta_{0} \right\},
\]
which is a complete metric space with respect to the distance $\rho(\eta_{1},\eta_{2})\vcentcolon=
\sup\{\|\eta_{1}(t,x)-\eta_{2}(t,x)\|:(t,x)\in[0,1]\times X\}$.
Observe that, by construction, $\eta([0,1]\times A_{\varepsilon})\subset D_{\Phi}$ for all
$\eta\in\mathcal{L}_{\varepsilon}$, that $(\{0\}\times X)\cup([0,1]\times B)\subset G_{\varepsilon}$
and, thus, $\eta(1,A_{\varepsilon})\in\mathcal{F}$ by hypothesis 1 of the theorem.
We define for $x\in X$:
\begin{align*}
f_{1}(x) & = \max\{0,\varepsilon^{2}-\varepsilon\dist(x,F)\} \\
f_{2}(x) & = \min\left\{\frac{\varepsilon^{2}}{8},\varepsilon\dist(x,(A_{\varepsilon}\setminus N_{\varepsilon}(F))\cup B)\right\} \\
g(x) & = I(x)+f_{1}(x)+f_{2}(x)
\end{align*}
and consider the functional $J:\mathcal{L}_{\varepsilon}\to\RR\cup\{+\infty\}$ given by
\[
J(\eta)=\sup_{z\in\eta(1,A_{\varepsilon})} g(z) = \sup_{x\in A_{\varepsilon}} g(\eta(1,x)),\quad \forall \ \eta \in \mathcal{L}_{\varepsilon}.
\]
Let us observe that $J$ is l.s.c. since it is the supremum of a family of l.s.c. functions.

\noindent
For every $\eta\in\mathcal{L}_{\varepsilon}$, there exists $z_{0}\in \eta(1, A_{\varepsilon}) \cap F$
by assumption 3, since $\eta(1,A_{\varepsilon})\in\mathcal{F}$, and, thus
\[
J(\eta)\ge g(z_{0}) \ge I(z_{0})+f_{1}(z_{0}) = I(z_{0}) + \varepsilon^{2}
\ge \inf_{F} I+\varepsilon^{2} = c+\varepsilon^{2}.
\]
Hence,
\[
c+\varepsilon^{2}\le\inf_{\eta\in\mathcal{L}_{\varepsilon}} J(\eta).
\]
On the other hand, by \eqref{eq:minseqA} we have:
\begin{equation}\label{eq:catena}
\inf_{\eta\in\mathcal{L}_{\varepsilon}} J(\eta) \le J(\overline{\eta}) =
\sup_{A_{\varepsilon}}(I+f_{1}+f_{2}) \le c +  \frac{\varepsilon^{2}}{8} + \varepsilon^{2} 
+ \frac{\varepsilon^{2}}{8} \le
\inf_{\eta\in\mathcal{L}_{\varepsilon}} J(\eta)+\frac{\varepsilon^{2}}{4}.
\end{equation}
By Ekeland's variational principle, there exists $\eta_{0}\in\mathcal{L}_{\varepsilon}$ such that
\begin{equation}\label{eq:EkVarPrinc}
J(\eta_{0})\le J(\overline{\eta}); \quad
\rho(\eta_{0},\overline{\eta})\le\frac{\varepsilon}{2}; \quad
J(\eta)\ge J(\eta_{0})-\frac{\varepsilon}{2}\rho(\eta,\eta_{0}),\;
\forall\eta\in\mathcal{L}_{\varepsilon}.
\end{equation}
We have that $\psi$ is bounded from above on $\eta_{0}(1,A_{\varepsilon})$, since
\[
\sup_{\eta_{0}(1,A_{\varepsilon})} \psi
\le \sup_{\eta_{0}(1,A_{\varepsilon})} I - \inf_{\eta_{0}(1,A_{\varepsilon})} \Phi
\le J(\eta_{0}) - \min_{\eta_{0}(1,A_{\varepsilon})} \Phi
\le J(\overline{\eta}) - \min_{\eta_{0}(1,A_{\varepsilon})} \Phi,
\]
and, thus, $\psi$ is continuous on $\eta_{0}(1,A_{\varepsilon})$ by Assumption~\ref{ass:I}.2.
In particular, the function $ x\mapsto g(\eta_{0}(1,x)) $ is continuous on $A_{\varepsilon}$ and
the set $C=\{ w\in\eta_{0}(1,A_{\varepsilon}):g(w)=\sup g(\eta_{0}(1,A_{\varepsilon})) \}$
is non-empty and compact.

\noindent
Now, let $\hat{z}\in\eta_{0}(1,A_{\varepsilon})\cap F$ satisfy
$ I(\hat{z}) = \sup I(\eta_{0}(1,A_{\varepsilon})\cap F) $: either $\hat{z}\not\in B$ or
we have
\[
I(\hat{z}) \le \sup_{B} I \le \inf_{F} I \le I(z) \le \sup _{\eta_{0}(1,A_{\varepsilon})\cap F} I,
\quad \forall \ z\in (\eta_{0}(1,A_{\varepsilon})\cap F) \setminus B.
\]
In both cases there is $z_{0}\in (\eta_{0}(1,A_{\varepsilon})\cap F) \setminus B$ such that
$I(z_{0})=\sup I(\eta_{0}(1,A_{\varepsilon})\cap F) $ and, thus,
\[
\max_{\eta_{0}(1,A_{\varepsilon})} g \ge g(z_{0}) = I(z_{0}) + \varepsilon^{2} + f_{2}(z_{0})
\ge c + \varepsilon^{2} + f_{2}(z_{0})> c + \varepsilon^{2},
\]
where we used the fact that $f_{2}(z_{0})>0$ and $ I(z_{0})\ge c$, since $z_{0}\in F$.

\noindent
On the other hand, we have that:
\begin{align*}
z\in A_{\varepsilon}\setminus N_{\varepsilon}(F) & \implies  g(z)=I(z) \le \sup_{A_{\varepsilon}} I< c+ \frac{{\varepsilon}^{2}}{8}; \\
z\in B &\implies  g(z)=I(z)+f_{1}(z) \le \sup_{B} I + \varepsilon^{2} \le c+\varepsilon^{2}.
\end{align*}
Hence, $ C \cap [(A_{\varepsilon}\setminus N_{\varepsilon}(F))\cup B] = \emptyset$.

\noindent
We make the following
\begin{description}
\item[Claim.] There exists $x_{\varepsilon}\in C $ such that:
$ \Phi^{0}(x_{\varepsilon};z-x_{\varepsilon}) + \psi(z) - \psi(x_{\varepsilon}) \ge
-5\varepsilon \|z-x_{\varepsilon}\|,\quad \forall \ z\in X$.
\end{description}
Assuming that the claim is proven, we have that $x_{\varepsilon}\not\in B$.
Moreover, $x_{\varepsilon}=\eta_{0}(1,\hat{x})$ for some $\hat{x}\in A_{\varepsilon}$:
in fact, $\hat{x} \in N_{\varepsilon}(F)$ because, otherwise, we would have
$\hat{x}\in A_{\varepsilon}\setminus N_{\varepsilon}(F)$ and, thus,
$x_{\varepsilon}=\hat{x}$, since $\eta_{0}\in\mathcal{L}_{\varepsilon}$, that is a contradiction
with $ C \cap [(A_{\varepsilon}\setminus N_{\varepsilon}(F))\cup B] = \emptyset$.
Therefore, we deduce that
\begin{align*}
\dist(x_{\varepsilon},A_{\varepsilon}) &\le \|\eta_{0}(1,\hat{x})-\hat{x}\|
\le \rho(\eta_{0},\overline{\eta})\le \frac{\varepsilon}{2}; \\
\dist(x_{\varepsilon},F) &\le \|\eta_{0}(1,\hat{x})-\overline{\eta}(1,\hat{x})\| +
\dist(\hat{x},F) \le \rho(\eta_{0},\overline{\eta})+\frac{\varepsilon}{2} \le \frac{3}{2}\varepsilon.
\end{align*}
Finally, from \eqref{eq:catena} we infer
\[
I(x_{\varepsilon}) \le g(x_{\varepsilon}) = J(\eta_{0}) \le J(\overline{\eta})
= \sup_{A_{\varepsilon}} g \le c + \frac{5}{4}\varepsilon^{2}
\]
and
\[
I(x_{\varepsilon}) = J(\eta_{0})-f_{1}(x_{\varepsilon})-f_{2}(x_{\varepsilon})
\ge \inf_{\mathcal{L}_{\varepsilon}} J -\varepsilon^{2}-\frac{\varepsilon^{2}}{8}
\ge c - \dfrac{\varepsilon^{2}}{8}.
\]
\bigskip

\noindent
Let us finally prove the remaining claim.
By contradition, if the claim is false, we have that \eqref{eq:nocritico} holds with $5\varepsilon$
in place of $\varepsilon$ and we can apply Theorem~\ref{thm:deformation} with
$B'=(A_{\varepsilon}\setminus N_{\varepsilon}(F))\cup B$ in place of $B$
and any $k\in(1,2)$.
We fix any $t_{1}\in(0,t_{0})$ such that $t_{1}\le \varepsilon/(2k)$, too, and 
define
\[
\eta_{\lambda}(t,x)=\alpha(\lambda t,\eta_{0}(t,x)),\quad \forall \ (t,x)\in[0,1]\times X, \ \forall \ \lambda\in[0,t_{1}].
\]
Using statements 1 and 2 of Theorem~\ref{thm:deformation} it is straightforward to check that
$\eta_{\lambda}(t,x)=x$ for all $(t,x)\in G_{\varepsilon}$ and all $\lambda\in[0,t_{1}]$.
Moreover, we have
\[
\|\eta_{\lambda}(t,x)-x\| \le \|\alpha(\lambda t,\eta_{0}(t,x))-\eta_{0}(t,x)\|
+\|\eta_{0}(t,x)-x\| \le k t_{1}+\frac{\varepsilon}{2} \le \delta_{1},
\quad \forall \ (t,x)\in[0,1]\times X, \forall \ \lambda\in[0,t_{1}],
\]
again by statement 2 of Theorem~\ref{thm:deformation} and the choiches for $t_{1}$ and $\epsilon$.

\noindent
Therefore we can deduce that $\eta_{\lambda}\in\mathcal{L}_{\varepsilon}$ and that
 $\rho(\eta_{\lambda},\eta_{0})\le \lambda k $, for every $\lambda\in[0,t_{1}]$.
 
\noindent
Hence, by \eqref{eq:EkVarPrinc} we have that
\begin{equation}\label{eq:36}
J(\eta_{\lambda}) \ge J(\eta_{0}) - \frac{e}{2}\rho(\eta_{\lambda},\eta_{0}) \ge
J(\eta_{0})-\frac{\varepsilon\lambda k}{2} \qquad \forall \lambda\in[0,t_{1}].
\end{equation}
By statement 3 in Theorem~\ref{thm:deformation}, we deduce that
\[
I(\eta_{\lambda}(1,x))\le I(\eta_{0}(1,x))-5\varepsilon\lambda\varphi(\eta_{0}(1,x)),
\quad \forall \ x\in A_{\varepsilon},\ \forall \ \lambda\in[0,t_{1}],
\]
and, thus,
\[
\sup_{\eta_{\lambda}(1,A_{\varepsilon})} I \le J(\eta_{0}) - 5\varepsilon\lambda\min_{\eta_{0}(1,A_{\varepsilon})} \varphi<+\infty.
\]
Arguing as above, this guarantees that $\psi$ is continuous on $\eta_{\lambda}(1,A_{\varepsilon})$
and that also the function $x\mapsto g(\eta_{\lambda}(1,x))$ is continuous on $A_{\varepsilon}$:
let $x_{\lambda}\in A_{\varepsilon}$ be such that $g(\eta_{\lambda}(1,x_{\lambda}))=J(\eta_{\lambda})$.
Inequality \eqref{eq:36} implies that
\begin{equation}\label{eq:37}
g(\eta_{\lambda}(1,x_{\lambda}))-g(\eta_{0}(1,x))\ge J(\eta_{0}) - g(\eta_{0}(1,x)) -
\frac{\varepsilon\lambda k}{2} \ge -\frac{\varepsilon\lambda k}{2},
\quad \forall \ x\in A_{\varepsilon},\ \forall \ \lambda\in[0,t_{1}].
\end{equation}
Now, recalling that the functions $f_{1}$ and $f_{2}$ have both Lipschitz constant $\varepsilon$,
we deduce that
\[
I(\eta_{\lambda}(1,x_{\lambda}))-I(\eta_{0}(1,x_{\lambda})) \ge g(\eta_{\lambda}(1,x_{\lambda}))-g(\eta_{0}(1,x_{\lambda})) -2\varepsilon \|\eta_{\lambda}(1,x_{\lambda})-\eta_{0}(1,x_{\lambda})\| \\
\ge-\frac{5}{2}\varepsilon\lambda k,
\]
for all $\lambda\in[0,t_{1}]$.
Moreover, again by statement 2 of Theorem~\ref{thm:deformation}, we have that
\[
I(\eta_{\lambda}(1,x_{\lambda}))-I(\eta_{0}(1,x_{\lambda})) \le
-5\varepsilon\lambda \varphi(\eta_{0}(1,x_{\lambda})).
\]
The last two inequalities together give that
\begin{equation}\label{eq:39}
\varphi(\eta_{0}(1,x_{\lambda})) \le \frac{k}{2} < 1.
\end{equation}
On the other hand, if we denote by $\hat{x}\in A_{\varepsilon}$ any cluster point of
$\{x_{\lambda}:\lambda\in[0,t_{1}]\}$ as $\lambda\to 0^{+}$, we observe that, again, the map
$(\lambda,x)\mapsto g(\eta_{\lambda}(1,x)) $ is continuous on the closure of
$\{(\lambda,x_{\lambda}):\lambda\in[0,t_{1}]\}\subset[0,t_{1}]\times A_{\varepsilon}$
and we let $\lambda\to 0^{+}$ in \eqref{eq:37} to obtain:
\[
g(\eta_{0}(1,\hat{x})) - g(\eta_{0}(1,x)) \ge 0,\quad \forall \ x \in A_{\varepsilon};
\]
this means that, actually, $\eta_{0}(1,\hat{x})\in C $ and, thus, $\varphi(\eta_{0}(1,\hat{x}))=1$
by statement 4 of Theorem~\ref{thm:deformation}.
This, however, contradicts \eqref{eq:39} and the claim is proved.
\subsection*{The second case: $\mathbf{\operatornamewithlimits{\mathbf{inf}}_{F} I < c}$.}
The argument in this case can be simplified since we do not seek a PS-sequence that
approaches $F$ anymore.
We set
\[
\mathcal{L}=\{\eta\in C([0,1]\times X,X) :\ \eta(t,x)=x,\quad \forall \ (t,x)\in(\{0\}\times X)\cup([0,1]\times B);\quad 
\rho(\eta,\overline{\eta})\le\delta_{0}\}.
\]
Again, we have that $(\mathcal{L},\rho)$ is a complete metric space and $\eta(1,A_{\varepsilon})\in\mathcal{F}$ for all $\eta\in\mathcal{L}$.
The functional $J:\mathcal{L}\to\left(-\infty,+\infty\right]$, given by
$J(\eta)=\sup f(\eta(1,A_{\varepsilon}))$, is l.s.c. and satisfies
\[
c\le\inf_{\mathcal{L}} J \le J(\overline{\eta})< c+\frac{\varepsilon^{2}}{4}
\le \inf_{\mathcal{L}} J +\frac{\varepsilon^{2}}{4}.
\]
Thus, there exists $\eta_{0}\in\mathcal{L}$ such that \eqref{eq:EkVarPrinc} holds
with $\mathcal{L}$ in place of $\mathcal{L}_{\varepsilon}$.
As in the previous case, $\psi$ is bounded from above on $\eta_{0}(1,A_{\varepsilon})$ and, using Assumption~\ref{ass:I}.2, we can consider
the nonempty and compact set $C=\{w\in\eta_{0}(1,A_{\varepsilon}): f(w)=\max I(\eta_{0}(1,A_{\varepsilon}))\}$.
Moreover, $B\cap C=\emptyset$ since
\[
\sup_{B} I \le \inf_{B} I <c\le\inf_{C} I
\]
and we can use again Theorem~\ref{thm:deformation} to show that there exists 
$x_{\varepsilon}\in C$ such that
\[
\Phi^{0}(x_{\varepsilon},z-x_{\varepsilon}) + \psi(z) - \psi(x_{\varepsilon}) \ge
-\varepsilon \|z-x_{\varepsilon}\|,\quad \forall  \ z\in X.
\]
The first inequality in \eqref{eq:EkVarPrinc} ensures that
\[
c\le \inf_{C} I \le I(x_{\varepsilon}) =J(\eta_{0})\le J(\overline{\eta}) <c+\frac{\varepsilon^{2}}{4},
\]
while the second one implies that $\dist(x_{\varepsilon},A_{\varepsilon})\le \varepsilon/2$.
\medbreak
\medbreak
\medbreak
\noindent
\textbf{Acknowledgments.} The authors wish to thank Susanna Terracini for some fruitful conversations and for her encouragment to pursue this research.


\begin{thebibliography}{99}
%
\bibitem{Abb01}
A. Abbondandolo, Morse theory for Hamiltonian systems, \textit{Chapman \& Hall/CRC Research Notes in Mathematics} \textbf{425},
	2001.
%
\bibitem{Abb03}
A. Abbondandolo, On the Morse index of Lagrangian systems, \textit{Nonlinear Anal.} \textbf{53} (2003), 551--566.
%
\bibitem{AnBa71}
C.M. Andersen \& H.C. von Baeyer, On classical scalar field theories and the relativistic Kepler problem, \textit{Ann. Physics} \textbf{62} (1971), 120--134.
%
\bibitem{ArBeTo20}
D. Arcoya, C. Bereanu \& P.J. Torres, 
Critical point theory for the Lorentz force equation, 
\textit{Arch. Ration. Mech. Anal.} \textbf{232} (2019), 1685--1724.
%
\bibitem{BeJeMa11}
C. Bereanu, P. Jebelean \& J. Mawhin, Variational methods for nonlinear perturbations of singular $\phi$-Laplacians,
\textit{Atti Accad. Naz. Lincei Rend. Lincei Mat. Appl.} \textbf{22} (2011), 89--111.
%
\bibitem{BeMa08}
C. Bereanu \& J. Mawhin, Boundary value problems for some nonlinear systems with singular $\phi$-Laplacian, \textit{J. Fixed Point Theory Appl.} \textbf{4} (2008), 57--75.
%
\bibitem{BoDaFe22}
A. Boscaggin, W. Dambrosio \& G. Feltrin, Periodic solutions to a
perturbed relativistic Kepler problem, \textit{SIAM J. Math. Anal.} \textbf{53} (2021), 5813--5834.
%
\bibitem{BoDaPa20}
A. Boscaggin, W. Dambrosio \& D. Papini, Periodic solutions to a forced Kepler problem in the plane, \textit{Proc. Amer. Math. Soc.}
\textbf{148} (2020), 301--314.
%
\bibitem{Bo04}
T.H. Boyer, Unfamiliar trajectories for a relativistic particle in a Kepler or Coulomb
potential, \textit{Amer. J. Phys.} \textbf{72} (2004), 992--997.
%
\bibitem{Chang81}
K.C. Chang, Variational methods for nondifferentiable functionals and their applications to partial differential equations,
\textit{J. Math. Anal. Appl.} \textbf{80} (1981), 102--129.
%
\bibitem{Clarke90}
F.H. Clarke, Optimization and nonsmooth analysis, \textit{Classics in Applied Mathematics} \textbf{5},
SIAM, Philadelphia, PA, 1990.
%
\bibitem{DeEr85}
E.A. Desloge \& E. Eriksen, Lagrange's equations of
motion for a relativistic particle, \textit{Amer. J. Phys.} \textbf{53} (1985), 83--84.
%
\bibitem{Ga19}
A.C. Gallo, Periodic solutions of perturbed central Hamiltonian systems, \textit{NoDEA
Nonlinear Differential Equations Appl.} \textbf{26} (2019) Art. 34, 24 pp.
%
\bibitem{Gh93}
N. Ghoussoub, A min-max principle with a relaxed boundary condition, \textit{Proc. Amer. Math.
Soc.} \textbf{117} (1993), 439--447.
%
\bibitem{GhPr89}
N. Ghoussoub \& D. Preiss, A general mountain pass principle for locating and classifyng
critical points, \textit{Ann. Inst. H. Poincar\'{e} Anal. Non Lin\'{e}aire} \textbf{6} (1989), 321--330.
%
\bibitem{GoPoSa02}
H. Goldstein, C. Poole \& J. Safko, Classical mechanics, Addison Wesley, San Francisco,
2002.
%
\bibitem{Go77}
W.B. Gordon, A minimizing property of Keplerian orbits, \textit{Am. J. Math.} \textbf{99} (1977), 961--971.
%
\bibitem{Ji13}
L. Jia, Approximate Kepler’s elliptic orbits with the relativistic effects, \textit{Int. J. Astron. Astrophys.} \textbf{3} (2013) 29--33.
%
\bibitem{KaOfPo21}
H. Kavle, D. Offin \& A. Portaluri, Keplerian orbits through the Conley-Zehnder index, \textit{Qual. Theory Dyn. Syst.} \textbf{20} (2021), 1--27.
%
\bibitem{LeMo-pp}
T.J. Lemmon \& A.R. Mondragon, Kepler’s orbits and Special Relativity in Introductory
Classical Mechanics, arXiv:1012.5438.
%
\bibitem{LiMa04}
R. Livrea \& S. Marano, Existence and classification of critical
points for nondifferentiable functions, 
\textit{Adv. Differential Equations} \textbf{9} (2004), 961--978.
%
\bibitem{MaWi89}
J. Mawhin \& M. Willem, Critical point theory and Hamiltonian systems, vol. 74 of
Applied Mathematical Sciences, Springer-Verlag, New York, 1989.
%
\bibitem{MoPa99}
D. Motreanu \& P.D. Panagiotopoulos,
Minimax theorems and qualitative properties of the solutions of hemivariational inequalities,
\textit{Nonconvex Optimization and its Applications} \textbf{29}, Kluwer Academic Publishers, Dordrecht, 1999.
%
\bibitem{MuPa06}
G. Mu\~{n}oz, I. Pavic, A Hamilton-like vector for the special-relativistic Coulomb problem,
\textit{European J. Phys.} \textbf{27} (2006), 1007--1018.
%
\bibitem{Sz86}
A. Szulkin, Minimax principles for lower semicontinuous functions and applications to nonlinear boundary value problems,
\textit{Ann. Inst. H. Poincar\'{e} Anal. Non Lin\'{e}aire} \textbf{3} (1986), 77--109.
%
\bibitem{ToUrZa13}
P.J. Torres, A.J. Ure\~{n}a, \& M. Zamora, Periodic and quasi-periodic motions of a relativistic
particle under a central force field, \textit{Bull. Lond. Math. Soc.} \textbf{45} (2013) 140--152.
%
\bibitem{Za13}
M. Zamora, New periodic and quasi-periodic motions of a relativistic particle under
a planar central force field with applications to scalar boundary periodic problems,
\textit{Electron. J. Qual. Theory Differ. Equ.} \textbf{31} (2013) 16 pp.

\end{thebibliography}
\end{document}